\documentclass{article}

\usepackage{algorithm}
\usepackage{algorithmic}
\usepackage{amsmath}
\usepackage{amssymb}
\usepackage{latexsym}
\usepackage{amsthm}
\usepackage{enumerate}
\usepackage{multirow}
\usepackage{ecltree}
\usepackage{mathrsfs}
\usepackage{mathtools}
\usepackage{stmaryrd}
\usepackage{xcolor}
\usepackage{comment}
\usepackage{tikz}
\usetikzlibrary{trees}
\usepackage[margin=25truemm]{geometry}

\newtheorem{definition}{Definition}
\newtheorem{proposition}{Proposition}
\newtheorem{lemma}{Lemma}
\newtheorem{theorem}{Theorem}
\newtheorem{corollary}{Corollary}
\newtheorem{remark}{Remark}
\newtheorem{example}{Example}

\newcommand{\st}{{\ast}}
\newcommand{\up}{{\uparrow}}
\newcommand{\down}{{\downarrow}}
\newcommand{\seq}{\mathbin{\to}}
\newcommand{\osum}{\mathbin{:}}

\newcommand{\short}{\tilde{\mathbb{G}}}
\newcommand{\Short}{\mathbb{G}}
\newcommand{\dicot}{\tilde{\mathbb{G}}^0}
\newcommand{\dyadic}{\mathbb{D}}
\newcommand{\birth}{\tilde{\mathrm{b}}}

\newcommand{\relmiddle}[1]{\mathrel{}\middle#1\mathrel{}}
\newcommand{\eqlab}[2]{\overset{(\mathrm{#1})}{#2}}
\newcommand{\defeq}[1]{\overset{\mathrm{def}}{#1}}

\begin{document}
\title{The Game Value of Sequential Compounds of Integers and Stars}  
\author{Kengo Hashimoto}

\date{University of Fukui, E-mail: khasimot@u-fukui.ac.jp}

\maketitle

\begin{abstract}
A combinatorial game is a two-player game without hidden information or chance elements.
One of the major approaches to analyzing games in combinatorial game theory is to break down a given game position into a \emph{disjunctive sum} of multiple sub-positions, then evaluate the \emph{game value} of each component of the sum, and finally integrate these game values to find which player has a winning strategy in the whole position.
Accordingly, finding the game value of a given position is a major topic in combinatorial game theory.
The sequential compound proposed by Stromquist and Ullman is a combinatorial game consisting of two combinatorial games.
In the sequential compound of games $G$ and $H$, the players make moves on $G$ until $G$ is over, and then they play on $H$.
In this paper, we investigate the general properties of sequential compounds.
As the main result, we give the game values of sequential compounds of a finite number of \emph{integers} and \emph{stars},
which are basic and typical games in combinatorial game theory.
\end{abstract}

\section{Introduction}

A \emph{combinatorial game} is a two-player game without hidden information or chance elements; e.g., Chess, Go, and Checkers.
In a combinatorial game, two players, conventionally called \emph{Left} and \emph{Right}, take turns to make a move alternately.
If it is guaranteed that the winner is determined within a finite number of moves,
then exactly one of the two players has a winning strategy; that is, by repeating the appropriate moves, the player can win regardless of the opponent's moves.
The main interest of combinatorial game theory is to determine which player has a winning strategy in a given combinatorial game.

One of the major approaches to analyzing a combinatorial game is to break down a given game position $G$ into a ``sum'' of multiple sub-positions $G_1 + G_2 + \cdots + G_n$,
then analyze each component $G_i$ independently, and finally integrate the results for individual components to obtain the result for the whole position $G$.
A \emph{disjunctive compound} of combinatorial games, in which a player makes a move on exactly one component, is the most commonly studied ``addition'' of combinatorial games because positions of many combinatorial games are naturally decomposed into a disjunctive compound of multiple sub-positions.
For a given disjunctive compound, we can find which player has a winning strategy by finding the \emph{game value} of each component and summing them up by algebraic operations on game values.
Therefore, finding the game value of a given position is a major topic in combinatorial game theory.

There are also other types of addition: in a \emph{conjunctive compound}, a player plays on every component;
in a \emph{selective compound}, a player plays on one or more arbitrary number of components \cite{Smi66}.
The \emph{ordinal sum} $G \osum H$ is a non-commutative addition of two combinatorial games $G$ and $H$,
in which a player plays on exactly one of $G$ and $H$, and if the player played on $G$, then $H$ is discarded keeping only $G$.
In recent research \cite{CM23}, Clow and McKay studied ordinal sums of \emph{numbers}, which are basic and typical games in combinatorial game theory.

A \emph{sequential compound} \cite{SU93} is another non-commutative addition of two combinatorial games.
In the sequential compound $G \seq H$ of two combinatorial games $G$ and $H$, the players make moves on $G$ until $G$ is over, and then they play on $H$.
Stromquist and Ullman \cite{SU93} studied who has a winning strategy and the game values of the sequential compound mainly for the class of games called \emph{impartial games}, 
in which the two players share the same set of options for every position
(in contrast, games in which the options of both players are not necessarily same are called \emph{partizan games}).
Stewart \cite{Ste07} investigated sequential compounds of partizan games under a slightly different definition from \cite{SU93}.
There is almost no research other than the above that explicitly deals with sequential compounds.

In this paper, we investigate the general properties of sequential compounds of partizan combinatorial games.
Also, as the main result, we give the game values of sequential compounds of a finite number of \emph{integers} and \emph{stars},
which are basic and typical games in combinatorial game theory.
This enables us to find who has a winning strategy in a given disjunctive compound of sequential compounds of integers and stars.

This paper is organized as follows.
In Section \ref{sec:preliminaries}, we introduce the basic definitions and notation used in this paper, along with some known results in combinatorial game theory.
Our main results are described in Section \ref{sec:main}.
Then in Section \ref{sec:conclusion}, we summarize our result and state the future work.
To clarify the flow of discussion, we defer some of the proofs to the appendix.

\section{Preliminaries}
\label{sec:preliminaries}

In this section, we introduce the definitions and notation used in this paper.
By the conventions, we refer to an individual game position as a \emph{game} rather than a system of playable rules such as Chess and Go\footnote{Instead, a system of playable rules is referred to as a \emph{ruleset} in combinatorial game theory.}.

\subsection{Short Games}

In this paper, we consider only \emph{short games}, in which the game ends after a finite number of moves, and every position has at most a finite number of options.
We briefly introduce the basic definitions and some known results in combinatorial game theory.
We leave the proofs of the propositions and other detailed discussion to the textbooks such as \cite{ANW19, BCG18, Con00, HG16, Sie13}.

Let $\short$ be the set of all short games.
As a player makes a move in a game $G \in \short$, the game $G$ transitions to another game $G' \in \short$ closer to the end than $G$.
A game $G \in \short$ is represented by an ordered pair of the set $\mathcal{L} \subseteq \short$ of possible transitions by  Left's moves and the set $\mathcal{R} \subseteq \short$ of possible transitions by Right's moves.
The formal definition is as follows.

\begin{definition}
The \emph{zero game} $\llbracket 0 \rrbracket$ is defined as the pair of two empty sets, that is, 
\begin{equation}
\label{eq:ue9mokgt6foc}
\llbracket 0 \rrbracket \defeq{=} (\emptyset, \emptyset).
\end{equation}
For an integer $n \geq 0$, the set $\short_n$ is defined as
\begin{equation*}
\short_n \defeq{=}
\begin{cases}
\{\llbracket 0 \rrbracket\} &\,\,\text{if}\,\,n = 0,\\
\{(\mathcal{G}^L, \mathcal{G}^R) : \mathcal{G}^L, \mathcal{G}^R \subseteq  \short_{n-1}\} &\,\,\text{if}\,\,n \geq 1.\\
\end{cases}
\end{equation*}
A \emph{short game} (or \emph{game}) is an element of
\begin{equation*}
\tilde{\mathbb{G}} \defeq{=} \bigcup_{n \geq 0} \tilde{\mathbb{G}}_n.
\end{equation*}
\end{definition}

Intuitively, the zero game $\llbracket 0 \rrbracket$ is the end position in which neither player can make any more moves.
Also, $\short_n$ is the set of all games that reach the zero game in at most $n$ moves even if both players take turns not necessarily alternately (i.e., one player can make consecutive moves).

By the definition, we have
\begin{equation*}
\short_0 \subsetneq \short_1 \subsetneq \short_2 \subsetneq \cdots.
\end{equation*}
For $G \in \short$, the smallest integer $n \geq 0$ satisfying $G \in \short_n$ is called the \emph{formal birthday} of $G$ and denoted by $\birth(G)$.

A game $(\mathcal{G}^L, \mathcal{G}^R) \in \tilde{\mathbb{G}}$ is conventionally denoted as
\begin{equation*}
\left\{ \mathcal{G}^L \relmiddle| \mathcal{G}^R \right\} \quad \text{or} \quad
\left\{ G^L_1, G^L_2, \ldots, G^L_l \relmiddle| G^R_1, G^R_2, \ldots, G^R_r \right\},
\end{equation*}
where $\mathcal{G}^L = \{G^L_1, G^L_2, \ldots, G^L_l\}$ and $\mathcal{G}^R = \{G^R_1, G^R_2, \ldots, G^R_r\}$.

An element of $\mathcal{G}^L$ (resp.~$\mathcal{G}^R$) is called a \emph{Left option} (resp.~\emph{Right option}) of $G$.
Let $G^{\mathcal{L}}\defeq{=} \mathcal{G}^L$ and $G^{\mathcal{R}} \defeq{=} \mathcal{G}^R$ for $G = \left\{ \mathcal{G}^L \relmiddle| \mathcal{G}^R \right\} \in \short$.
Namely, $G^{\mathcal{L}}$ (resp.~$G^{\mathcal{R}}$) denotes the set of all Left (resp.~Right) options of $G$.
If $G, H \in \short$ are the same elements of $\short$, then we write $G \cong H$ instead of $G = H$.
A notation $G = H$ is reserved by another meaning as stated later, and $G = H$ and $G \cong H$ are strictly distinguished.

The two players take turns alternately until one of the players cannot make any more moves on the player's turn.
There are two conventions for determining the winner: \emph{normal play} and \emph{mis\`{e}re play}.
In normal (resp.~mis\`{e}re) play, the player who cannot move on the player's turn loses (resp.~wins).
The normal play convention is more actively studied because of its rich algebraic structure and partial order structure.
This paper also focuses on the normal play.
However, the discussion in this paper includes the case of mis\`{e}re play in the sense stated later in Remark \ref{rem:seq}.

For each game, it is determined which player has a winning strategy in the case where Left plays first and the case where Right plays first, respectively.
Namely, $\short$ is divided into a disjoint union $\short = \mathscr{N}^L \sqcup \mathscr{P}^R$,
where 
$\sqcup$ denotes a disjoint union of sets,
and $\mathscr{N}^L$ (resp.~$\mathscr{P}^R$) is the set of all games in which Left (resp.~Right) wins when Left plays first\footnote{The symbol $\mathscr{N}^L$ means that ``L''eft wins as the ``N''ext player. Also, $\mathscr{P}^R$ means that ``R''ight player wins as the ``P''revious player, that is, Right wins when Left plays first. These symbols are taken from \cite{Sie13}.}.
Similarly, we have a decomposition $\short = \mathscr{P}^L \sqcup \mathscr{N}^R$, where $\mathscr{P}^L$ (resp.~$\mathscr{N}^R$) is the set of all games in which Left (resp.~Right) wins when Right plays first.

Therefore, 
according to which player wins when Left plays first and when Right plays first,
the games are classified into four sets $\mathscr{L}, \mathscr{R}, \mathscr{N}, \mathscr{P}$ defined as follows.
\begin{itemize}
\item $\mathscr{L}$: regardless of who plays first, Left wins.
\item $\mathscr{R}$: regardless of who plays first, Right wins.
\item $\mathscr{N}$: the player who plays first (i.e., the Next player) wins.
\item $\mathscr{P}$: the player who plays second (i.e., the Previous player) wins.
\end{itemize}

The formal definition is as follows.

\begin{definition}
\label{def:outcome}
We define $\mathscr{P}^L, \mathscr{N}^L, \mathscr{P}^R, \mathscr{N}^R \subseteq \short$ as
the unique sets such that for any $G \in \short$, the following conditions (a)--(d) hold.
\begin{enumerate}[(a)]
\item $G \in \mathscr{P}^L$ if for any $G^R \in G^{\mathcal{R}}$, it holds that $G^R \in \mathscr{N}^L$.
\item $G \in \mathscr{N}^L$ if for some $G^L \in G^{\mathcal{L}}$, it holds that $G^L \in \mathscr{P}^L$.
\item $G \in \mathscr{P}^R$ if for any $G^L \in G^{\mathcal{L}}$, it holds that $G^L \in \mathscr{N}^R$.
\item $G \in \mathscr{N}^R$ if for some $G^R \in G^{\mathcal{R}}$, it holds that $G^R \in \mathscr{P}^R$.
\end{enumerate}
The sets $\mathscr{L}, \mathscr{R}, \mathscr{N}, \mathscr{P}$, called \emph{outcome classes}, are defined as follows.
\begin{eqnarray*}
\mathscr{L} &\defeq{=}& \mathscr{P}^L \cap \mathscr{N}^L,\\
\mathscr{R} &\defeq{=}& \mathscr{P}^R \cap \mathscr{N}^R,\\
\mathscr{P} &\defeq{=}& \mathscr{P}^L \cap \mathscr{P}^R,\\
\mathscr{N} &\defeq{=}& \mathscr{N}^L \cap \mathscr{N}^R.
\end{eqnarray*}
\end{definition}

Note that $\llbracket 0 \rrbracket \in \mathscr{P}^L \cap \mathscr{P}^R = \mathscr{P}$
since $\llbracket 0 \rrbracket$ satisfies Definition \ref{def:outcome} (a) and (c) by $\llbracket 0 \rrbracket^{\mathcal{L}} = \llbracket 0 \rrbracket^{\mathcal{R}} = \emptyset$.
This is consistent with the fact that the next player cannot move and loses in $\llbracket 0 \rrbracket$ due to normal play.

Every game belongs to exactly one outcome class, that is, $\short = \mathscr{L} \sqcup \mathscr{R} \sqcup \mathscr{P} \sqcup \mathscr{N}$.
Table \ref{tab:order-outcome} shows the relations among the outcome classes, where the inequalities such as $G \geq 0$ in the table are used later.
For $G \in \short$, the outcome class to which $G$ belongs is denoted by $o(G)$.

The \emph{disjunctive sum} $G+H$ of games $G$ and $H$ is a game made by combining $G$ and $H$
in which a player makes a move on exactly one of $G$ and $H$.

\begin{definition}
For $G, H \in \short$,
the \emph{disjunctive sum} (or \emph{disjunctive compound}) $G + H \in \short$ of $G$ and $H$ is defined recursively as
\begin{eqnarray*}
(G+H)^{\mathcal{L}} &\defeq{=}& (G^{\mathcal{L}} + H) \cup (G + H^{\mathcal{L}}),\\
(G+H)^{\mathcal{R}} &\defeq{=}& (G^{\mathcal{R}} + H) \cup (G + H^{\mathcal{R}}),
\end{eqnarray*}
where $\mathcal{G} + H \defeq{=}  H + \mathcal{G} \defeq{=}  \{G + H : G \in \mathcal{G}\}$ for $\mathcal{G} \subseteq \short$ and $H \in \short$.
\end{definition}

The disjunctive sum is associative and commutative and has the identity $\llbracket 0 \rrbracket \in \short$ as follows.

\begin{proposition}
The following statements (i)--(iii) hold.
\begin{enumerate}[(i)]
\item For any $G, H, J\in \short$, we have $(G + H) + J \cong G + (H + J)$.
\item For any $G, H \in \short$, we have $G + H \cong H + G$.
\item For any $G \in \short$, we have $G + \llbracket 0 \rrbracket \cong G$.
\end{enumerate}
\end{proposition}

To compare the advantages of two games for a player, we introduce an order relation $\leq$ between games
in such a way that the more advantageous a game $G$ is for Left, the greater the game $G$ is in the order $\leq$.

We first define a partial order of the set $\{\mathscr{L}, \mathscr{R}, \mathscr{P}, \mathscr{N}\}$ of all outcome classes
in such a way that the more desirable one for Left is greater:
$\mathscr{P} \leq \mathscr{L}$ and $\mathscr{N} \leq \mathscr{L}$ since $\mathscr{L}$ is more desirable than $\mathscr{P}$ and $\mathscr{N}$ for Left;
$\mathscr{R} \leq \mathscr{P}$ and $\mathscr{R} \leq \mathscr{N}$ since $\mathscr{P}$ and $\mathscr{N}$ are more desirable than $\mathscr{R}$ for Left;
$\mathscr{N}$ and $\mathscr{P}$ are incomparable to each other.
Then a game $G$ is less than or equal to a game $H$ if and only if the outcome class of $G + X$ is less than or equal to the outcome class of $H + X$ for any game $X$.

\begin{definition}
We define a binary relation $\leq$ of $\short$ as follows:
for any $G, H \in \short$, it holds that $G \leq H$ if and only if
\begin{equation*}
\forall X \in \short,\,\, o(G+X) \leq o(H+X),
\end{equation*}
where two outcome classes $o_1, o_2 \in  \{\mathscr{L}, \mathscr{R}, \mathscr{P}, \mathscr{N}\}$ are defined to satisfy $o_1 \leq o_2$ if and only if
\begin{equation*}
(o_1, o_2) \in \{(\mathscr{L}, \mathscr{L}), (\mathscr{R}, \mathscr{R}), (\mathscr{P}, \mathscr{P}), (\mathscr{N}, \mathscr{N}), 
(\mathscr{R}, \mathscr{L}), (\mathscr{R}, \mathscr{P}), (\mathscr{R}, \mathscr{N}), (\mathscr{P}, \mathscr{L}), (\mathscr{N}, \mathscr{L})\}.
\end{equation*}
\end{definition}

We write $G = H$ (resp.~$G < H$, $G \parallel H$) if $G \leq H$ and $G \geq H$ (resp.~if $G \leq H$ and $G \not\geq H$, if $G \not\leq H$ and $G \not\geq H$).

The binary relation $\leq$ of $\short$ satisfies the reflexivity and transitivity, that is, the binary relation $\leq$ is a preorder of $\short$.
Also, the order by $\leq$ is preserved by the disjunctive sum.
Namely, the following proposition holds.

\begin{proposition}
\label{prop:order}
The following statements (i)--(iii) hold.
\begin{enumerate}[(i)]
\item For any $G \in \short$, we have $G \leq G$.
\item For any $G, H, J \in \short$, if $G \leq H$ and $H \leq J$, then $G \leq J$.
\item  For any $G, H, J \in \short$, if $G \leq H$, then $G+J \leq H+J$.
In particular, if $G = H$, then $G+J = H+J$.
\end{enumerate}
\end{proposition}

The order relation and the outcome classes correspond as follows.
\begin{proposition}
\label{prop:order-outcome}
For any $G \in \short$, the following statements (i)--(iv) hold.
\begin{enumerate}[(i)]
\item $G = \llbracket 0 \rrbracket \iff o(G) = \mathscr{P}$.
\item $G > \llbracket 0 \rrbracket \iff o(G) = \mathscr{L}$.
\item $G < \llbracket 0 \rrbracket \iff o(G) = \mathscr{R}$.
\item $G \parallel \llbracket 0 \rrbracket \iff o(G) = \mathscr{N}$.
\end{enumerate}
\end{proposition}
Table \ref{tab:order-outcome} shows the correspondence between the outcome classes and the order relation.

By Proposition \ref{prop:order-outcome}, we can replace the discussion on the outcome classes with that on the order relations, and vice versa.
For example, to prove that Left can win a game $G$ playing first (i.e., $G \in \mathscr{L} \cup \mathscr{N} = \mathscr{N}^L)$, it suffices to prove $G \not\leq \llbracket 0 \rrbracket$.

\begin{table}
\label{tab:order-outcome}
    \centering
    \caption{The correspondence between the outcome classes and the order relation}
    \label{tab:hogehoge}
    \begin{tabular}{cccc}
        \multicolumn{2}{c}{\multirow{4}{*}{}} & \multicolumn{2}{|c}{If Right plays first} \\ \cline{3-4}
        & &\multicolumn{1}{|c|}{Left wins} & Right wins\\
        & &\multicolumn{1}{|c|}{$\mathscr{P}^L$} & $\mathscr{N}^R$\\
        & &\multicolumn{1}{|c|}{$G \geq \llbracket 0 \rrbracket$} & $G \not\geq \llbracket 0 \rrbracket$\\
        \hline
        \multirow{6}{*}{\begin{tabular}{p{10pt}}If Left plays first\end{tabular}}
        & \multicolumn{1}{|c|}{Left wins} & \multicolumn{1}{c|}{Left wins} & First player wins\\
        & \multicolumn{1}{|c|}{$\mathscr{N}^L$} & \multicolumn{1}{c|}{$\mathscr{L}$} & $\mathscr{N}$ \\
        & \multicolumn{1}{|c|}{$G \not\leq \llbracket 0 \rrbracket$} & \multicolumn{1}{c|}{$G > \llbracket 0 \rrbracket$} & $G \parallel \llbracket 0 \rrbracket$ \\ \cline{2-4}
        & \multicolumn{1}{|c|}{Right wins} & \multicolumn{1}{c|}{Second player wins} & Right wins\\
        & \multicolumn{1}{|c|}{$\mathscr{P}^R$} & \multicolumn{1}{c|}{$\mathscr{P}$} & $\mathscr{R}$ \\
        & \multicolumn{1}{|c|}{$G \leq \llbracket 0 \rrbracket$} & \multicolumn{1}{c|}{$G = \llbracket 0 \rrbracket$} & $G < \llbracket 0 \rrbracket$ \\
    \end{tabular}
\end{table}

The binary relation $=$ is an equivalence relation of $\short$ by Proposition \ref{prop:order} (i) and (ii).
For $G \in \short$, the \emph{game value} of $G$, denoted by $\bar{G}$, is defined as the equivalence class modulo $=$ to which $G$ belongs.
Note that $G = H$ if and only if
\begin{equation*}
\forall X \in \short,\,\, o(G+X) = o(H+X).
\end{equation*}
Namely, if $G$ and $H$ have the same game value, then $G$ in a disjunctive sum can be interchanged with $H$ without changing the outcome class of the whole disjunctive sum.

Let $\Short$ denote the quotient set by the equivalence relation $=$, that is, $\Short \defeq= \short / {=}$.
The order relation and disjunctive sum of $\short$ can be naturally extended to $\Short$ as
\begin{eqnarray}
\bar{G} \leq \bar{H} &\defeq{\iff}& G \leq H, \label{eq:xqzkskc9vool}\\
\bar{G} + \bar{H} &\defeq{=}& \overline{G+H} \label{eq:5bkgaabx6q3e}
 \end{eqnarray}
 for $G, H \in \short$.
By Proposition \ref{prop:order} (iii), the above definitions (\ref{eq:xqzkskc9vool}) and (\ref{eq:5bkgaabx6q3e}) are well-defined.
The order relation $\leq$ of $\Short$ is a partial order since it also satisfies the anti-symmetry: $\bar{G} \leq \bar{H}$ and $\bar{H} \leq \bar{G}$ imply $\bar{G} = \bar{H}$.

This yields the following corollary of Proposition \ref{prop:order-outcome}.
\begin{corollary}
\label{cor:order-outcome}
For any $G \in \short$, the following statements (i)--(iv) hold.
\begin{enumerate}[(i)]
\item $\bar{G} = \overline{\llbracket 0 \rrbracket} \iff o(G) = \mathscr{P}$.
\item $\bar{G} > \overline{\llbracket 0 \rrbracket} \iff o(G) = \mathscr{L}$.
\item $\bar{G} < \overline{\llbracket 0 \rrbracket} \iff o(G) = \mathscr{R}$.
\item $\bar{G} \parallel \overline{\llbracket 0 \rrbracket} \iff o(G) = \mathscr{N}$.
\end{enumerate}
\end{corollary}

Therefore, to find the outcome class of a given disjunctive sum $G \defeq{\cong} G_1 + G_2 + \cdots + G_n$,
it suffices to find the game value $\bar{G_i}$ of each component $G_i$ and
evaluate the order relationship between $\overline{\llbracket 0 \rrbracket}$ and the sum
\begin{equation}
\bar{G} = \overline{G_1 + G_2 + \cdots + G_n} = \bar{G_1} + \bar{G_2} + \cdots + \bar{G_n}
\end{equation}
by using Corollary \ref{cor:order-outcome}.
Further, we may choose a game that is simple and easy to handle as the representative $G_i$ of each term.
In many cases, ``finding the game value of a game $G$'' means finding a sufficiently simple representative of $\bar{G}$
\footnote{It is known that every $\bar{G} \in \Short$ has the ``simplest'' representative called the \emph{canonical form} of $\bar{G}$.
More precisely, there exists a unique game $G^{\st} \in \bar{G}$ with the minimum formal birthday in $\bar{G}$.}.

Note that $\Short$ forms a commutative monoid with the disjunctive sum defined in (\ref{eq:5bkgaabx6q3e}).
Further, the set $\Short$ forms an abelian group by the following Definition \ref{def:negation} of the negative $-G \in \short$ of a game $G \in \short$.
More precisely, $\overline{-G} \in \Short$ serves as the inverse of $\bar{G} \in \Short$ as shown in Proposition \ref{prop:negation}.
The negative $-G$ is defined by exchanging the roles of Left and Right in $G$ as follows.

\begin{definition}
\label{def:negation}
For $G \in \short$, the \emph{negative} $-G$ of $G$ is defined recursively as
\begin{equation}
\label{eq:sppm6ebdbfgf}
-G \defeq{\cong} \left\{ -(G^{\mathcal{R}}) \relmiddle| -(G^{\mathcal{L}}) \right\},
\end{equation}
where $-\mathcal{G} \defeq{=} \{-G : G \in \mathcal{G}\}$ for $\mathcal{G} \subseteq \short$.
\end{definition}

\begin{proposition}
\label{prop:negation}
The following statements (i)--(iii) hold.
\begin{enumerate}[(i)]
\item For any $G \in \short$, we have $-(-G) \cong G$.
\item For any $G \in \short$, we have $G + (-G) = \llbracket 0 \rrbracket$.
\item For any $G, H \in \short$, we have $-(G+H) = (-G) + (-H)$.
\end{enumerate}
\end{proposition}

We use $G-H$ as a shorthand for $G + (-H)$.

Since the disjunctive sum is associative and commutative, we may use the following notation without ambiguity.
For an integer $n$ and $G \in \short$, let $n \cdot G$ denote the disjunctive sum of $n$ copies of $G$.
More precisely,
\begin{equation*}
n \cdot G \defeq{\cong}
\begin{cases}
\llbracket 0 \rrbracket &\,\,\text{if}\,\, n = 0, \\
\underbrace{G + G + \cdots + G}_n &\,\,\text{if}\,\, n > 0,\\
\underbrace{(-G) + (-G) + \cdots + (-G)}_{-n} &\,\,\text{if}\,\, n < 0.\\
\end{cases}
\end{equation*}
Also, for $G_1, G_2, \ldots, G_n \in \short$, let
\begin{equation*}
\sum_{i = 1}^n G_i \defeq{\cong} G_1 + G_2 + \cdots + G_n.
\end{equation*}

Replacing an option of $G$ with a greater one makes $G$ greater as follows.

\begin{proposition}[Replacement Lemma]
For any $G \cong \{G^L_1, \ldots, G^L_l \mid G^R_1, \ldots, G^R_r \} \in \short$, the following statements (i) and (ii) hold.
\begin{enumerate}[(i)]
\item Let $G' \cong \{G'^L_1, \ldots, G^L_l \mid G^R_1, \ldots, G^R_r \}$ be the game obtained by replacing the Left option $G^L_1$ of $G$ with a game $G'^L_1$ such that $G'^L_1 \geq G^L_1$. Then we have $G' \geq G$.
\item Let $G' \cong \{G^L_1, \ldots, G^L_l \mid G'^R_1, \ldots, G^R_r \}$ be the game obtained by replacing the Right option $G^R_1$ of $G$ with a game $G'^R_1$ such that $G'^R_1 \geq G^R_1$. Then we have $G' \geq G$.
\end{enumerate}
\end{proposition}

Combining with the symmetric argument, we obtain the following corollary: we can replace an option of $G$ with another with an equal game value without changing the game value.

\begin{corollary}
\label{cor:replace}
For any $G \cong \{G^L_1, \ldots, G^L_l \mid G^R_1, \ldots, G^R_r \} \in \short$, the following statements (i) and (ii) hold.
\begin{enumerate}[(i)]
\item Let $G' \cong \{G'^L_1, \ldots, G^L_l \mid G^R_1, \ldots, G^R_r \}$ be the game obtained by replacing the Left option $G^L_1$ of $G$ with a game $G'^L_1$ such that $G'^L_1 = G^L_1$. Then we have $G' = G$.
\item Let $G' \cong \{G^L_1, \ldots, G^L_l \mid G'^R_1, \ldots, G^R_r \}$ be the game obtained by replacing the Right option $G^R_1$ of $G$ with a game $G'^R_1$ such that $G'^R_1 = G^R_1$. Then we have $G' = G$.
\end{enumerate}
\end{corollary}

Let $\mathbb{Z}$ denote the set of all integers,
and define the set $\dyadic$ of all \emph{dyadic rationals} as
\begin{equation*}
\dyadic \defeq{=} \left\{\frac{m}{2^n} : m, n \in \mathbb{Z}, n \geq 0\right\}.
\end{equation*}
For $x \in \dyadic$, the corresponding game $\llbracket x \rrbracket$ is defined recursively as follows.

\begin{definition}
\label{def:number}
For $x \in \dyadic \setminus \{0\}$, the game $\llbracket x \rrbracket$ is defined recursively as the following statements (i)--(iii).
Note that $\llbracket 0 \rrbracket$ is already defined by (\ref{eq:ue9mokgt6foc}).
\begin{enumerate}[(i)]
\item For an integer $n \geq 1$, the game $\llbracket n \rrbracket$ is defined recursively as
$\llbracket n \rrbracket \defeq{\cong}  \{ \llbracket n-1 \rrbracket \mid \}$.

\item For an odd integer $m \geq 1$ and an integer $n \geq 1$, the game $\left\llbracket \frac{m}{2^n}\right\rrbracket$ is defined recursively as
\begin{equation}
\label{eq:g0h3zqunv24n}
\left\llbracket \frac{m}{2^n}\right\rrbracket \defeq{\cong}  \left\{ \left\llbracket \frac{m-1}{2^n}\right\rrbracket \relmiddle| \left\llbracket\frac{m+1}{2^n}\right\rrbracket \right\}.
\end{equation}

\item For $x \in \dyadic$ such that $x > 0$, the game $\llbracket -x \rrbracket$ is defined as
$\llbracket -x \rrbracket \defeq{\cong} - \llbracket x \rrbracket$.
\end{enumerate}
\end{definition}

Intuitively, for an integer $n \geq 0$, the game $\llbracket n \rrbracket$ (resp.~$\llbracket -n \rrbracket$) is the position in which only Left (resp.~Right) can make exactly $n$ moves.

In disjunctive sums and the order relation of games, the game $\llbracket x \rrbracket$ for $x \in \dyadic$ behaves in the same way as the ordinary addition and total order of real numbers as follows.

\begin{proposition}
\label{prop:number}
The following statements (i)--(iii) hold.
\begin{enumerate}[(i)]
\item For any $x, y \in \dyadic$, we have $\llbracket x+y \rrbracket = \llbracket x \rrbracket + \llbracket y \rrbracket$.
\item For any $x, y \in \dyadic$, we have $x \leq y \iff \llbracket x \rrbracket \leq \llbracket y \rrbracket$.
\item For any $x \in \dyadic$, we have $\llbracket -x \rrbracket \cong - \llbracket x \rrbracket$.
\end{enumerate}
\end{proposition}

Because of Proposition \ref{prop:number}, we can identify a game $\llbracket x \rrbracket \in \short$ with the dyadic rational $x \in \dyadic$.
In this sense, we have $\mathbb{Z} \subsetneq \dyadic \subsetneq \short$.
Accordingly, we simply write $\llbracket x \rrbracket$ as $x$ by an abuse of notation as long as there is no risk of misunderstanding.

A game in $\dyadic$ is called a \emph{number}.
We also refer to a game $X$ such that $X = x$ for some $x \in \dyadic$ as a number.
We can quantify how advantageous a game $G$ is for a player by comparing $G$ with numbers, 
and thus the numbers serve as a measure of the advantage of games.

By Definition \ref{def:number}, for any $x \in \dyadic$ and its Left option $x^L$ (resp.~Right option $x^R$), 
it holds that
\begin{equation}
\label{eq:d21j3vzuqcup}
x^L < x \quad \text{(resp.}~x < x^R\text{)}.
\end{equation}
Namely, playing on $x$ makes the situation more disadvantageous for the player,
and thus both players should avoid playing on $x$ as possible.
It is known as the \emph{Number Avoidance Theorem} that if a player has a winning strategy in a disjunctive sum $X + G$ of a number $X$ and a non-number $G$, then there is a winning move that plays on $G$.
Therefore, there is no need to consider the options that play on $X$.
This yields the following theorem, known as the \emph{Number Transition Theorem}.

\begin{proposition}[Number Translation Theorem]
\label{prop:number-transition}
For any $G, X \in \short$, if $X = x$ for some $x \in \mathbb{D}$ and $G$ is not, then
$G + X = \left\{ G^{\mathcal{L}} + X \relmiddle| G^{\mathcal{R}} + X \right\}$.
\end{proposition}

Next, we introduce a special class $\dicot$ of games called \emph{dicotic games}, in which one player has options if and only if the other player also does.
The formal definition is as follows.

\begin{definition}
The set $\dicot$ is defined as the smallest subset $\mathcal{G}$ of $\short$ satisfying the following conditions (a) and (b).
\begin{enumerate}[(a)]
\item $0 \in \mathcal{G}$.
\item For any $G \in \short$, if $\emptyset \neq G^{\mathcal{L}} \subseteq \mathcal{G}$ and $\emptyset \neq G^{\mathcal{R}} \subseteq \mathcal{G}$, then $G \in \mathcal{G}$.
\end{enumerate}
An element of $\dicot$ is said to be \emph{dicotic} (or \emph{all-small}).
\end{definition}

A dicotic game has an ``infinitesimal'' game value as follows.

\begin{proposition}[Lawnmower Theorem]
\label{prop:lawnmower}
For any $G \in \dicot$ and positive $x \in \dyadic$, we have $-x < G < x$.
\end{proposition}

As important examples of dicotic games, we introduce a \emph{star} and \emph{uptimals} in the following Definitions \ref{def:star} and \ref{def:uptimal}.

\begin{definition}
\label{def:star}
A game $\st \in \short$, called a \emph{star}, is defined as $\st \defeq{\cong} \{ 0 \mid 0 \}$.
\end{definition}

The star $\ast$ satisfies the following property directly from the definition.

\begin{proposition}
\label{prop:star}
The star $\st$ is dicotic and satisfies $\st \parallel 0$ and $\st + \st = 0$.
\end{proposition}

\begin{definition}
\label{def:uptimal}
For any integer $k \geq 1$, the game $\up^k$, called an \emph{up-$k$th}, is defined recursively as
\begin{equation*}
\up^k \defeq{\cong} \left\{ 0 \relmiddle| \st - \sum_{i = 1}^{k-1} \up^i \right\}.
\end{equation*}
A game in the form $\sum_{i = 1}^k a_i \up^i$ for some integer $k \geq 1$ and a sequence $(a_1, a_2, \ldots, a_k)$ of integers is called an \emph{uptimal}.
The game $\up^1 \cong \left\{0 \relmiddle| \st \right\}$ is simply denoted by $\up$ and called an \emph{up}.
Also, $\down^k$ (resp.~$\down$) denotes $-\up^k$ (resp.~$-\up$) and is called a \emph{down-$k$th} (resp.~\emph{down}), that is,
\begin{equation}
\label{eq:9f1ueuce84aw}
\down^k \defeq{\cong} \left\{  \st - \sum_{i = 1}^{k-1} \down^i \relmiddle|  0\right\}.
\end{equation}
\end{definition}

The following Propositions \ref{prop:uptimal-comp} and \ref{prop:uptimal-star} give the order relation among uptimals and a star.

\begin{proposition}
For any integer $k \geq1$, the following statements (i) and (ii) hold.
\label{prop:uptimal-comp}
\begin{enumerate}[(i)]
\item For any positive $x \in \dyadic$, we have $0 < \up^k < x$.
\item For any integer $n \geq 0$, we have $n \cdot \up^{k+1} < \up^k$.
\end{enumerate}
\end{proposition}

Proposition \ref{prop:uptimal-comp} is expressed concisely as
\begin{equation}
\label{eq:6edoguogaafz}
0 \ll \cdots \ll \up^3 \ll \up^2 \ll \up \ll 1,
\end{equation}
where $G \ll H$ represents that $n\cdot G < H$ holds for any integer $n \geq 0$.
In particular, an up-$k$th has an infinitesimally positive game value.

\begin{proposition}
\label{prop:uptimal-star}
For any integer $k \geq 1$, we have
\begin{equation*}
\sum_{i = 1}^k \up^i \parallel \st < \sum_{i = 1}^k \up^i + \up^k.
\end{equation*}
\end{proposition}

The following notation called \emph{uptimal notation} is useful to express the relation in (\ref{eq:6edoguogaafz}) intuitively:
\begin{eqnarray*}
0.a_1a_2\ldots a_k &\defeq{\cong}& \sum_{i = 1}^k a_i \cdot \up^k,
\end{eqnarray*}
where $a_1, \ldots, a_k$ are non-negative integers.
For example, $\up^2 + \up^2 + \up^4 + \up^4 + \up^4$ is denoted as $0.0203$ in the uptimal notation,
and Proposition \ref{prop:uptimal-star} is expressed as $0.11{\ldots} \parallel \st < 0.11{\ldots}12$ intuitively.

The game values appearing in this paper are limited to the form of
$x + \sum_{i = 1}^k b_i \cdot \up^i + c\cdot\st$
for some $x \in \dyadic$, an integer $k \geq 0$, $c \in \{0, 1\}$, and $b_1, b_2, \ldots, b_k \in \mathbb{Z}$.
Note that $n \cdot \st = (n \bmod 2) \cdot \st$ for any integer $n$.
Regarding this form, the following proposition holds.

\begin{proposition}
\label{prop:uptimal-unique}
For any integer $k \geq 0$, $x, x' \in \dyadic$, $c, c' \in \{0, 1\}$ and two sequences $(b_1, b_2, \ldots, b_k)$ and $(b'_1, b'_2, \ldots, b'_k)$ of integers,
the games 
$G \defeq\cong x + c\cdot\st + \sum_{i = 1}^k b_i \cdot \up^i$ and 
$G' \defeq\cong x' + c'\cdot\st + \sum_{i = 1}^k b'_i \cdot \up^i$
satisfy $G = G'$ if and only if
$(x, c, b_1, b_2, \ldots, b_k) = (x', c', b'_1, b'_2, \ldots, b'_k)$.
\end{proposition}

\begin{remark}
The outcome class of a game $G$ in the form of Proposition \ref{prop:uptimal-unique} can be determined by considering the following three cases: the case $x > 0$, the case $x < 0$, and the case $x \cong 0$.
\begin{itemize}
\item The case $x > 0$: Since $c \cdot \st + \sum_{i=1}^k b_i\cdot \up^i$ is dicotic, we have $G = x + c \cdot \st + \sum_{i=1}^k b_i\cdot \up^i > 0$ by Proposition \ref{prop:lawnmower}, and thus $o(G) = \mathscr{L}$.
\item The case $x < 0$: By the symmetric argument to the case $x > 0$, we obtain $G < 0$ and thus $o(G) = \mathscr{R}$.
\item The case $x \cong 0$: If $c = b_1 = b_2 = \cdots = b_n = 0$, then $G = 0$ and thus $o(G) = \mathscr{P}$.
If not, we can determine $o(G) = \mathscr{N}$, $o(G) = \mathscr{L}$, or $o(G) = \mathscr{R}$ by Propositions \ref{prop:uptimal-comp} and \ref{prop:uptimal-star}.
\end{itemize}
\end{remark}

In addition to the general definitions and results of combinatorial game theory described above,
we introduce a definition and a lemma unique to this paper to state our main results.

\begin{definition}
\label{def:uptimal-deg}
Let $G$ be a game such that
$G \cong c\cdot\st + \sum_{i = 1}^k b_i \cdot\down^i$
for some integer $k \geq 0$ and sequence $(c, b_1, b_2, \ldots, b_k)$ of non-negative integers.
Then we define $\deg(G)$ as the maximum integer $i \geq 1$ such that $b_i \geq 1$,
where $\deg(0) \defeq{=} \deg(\st) \defeq{=} 0$.
\end{definition}

\begin{lemma}
\label{lem:uptimal-option}
Let $G$ be a game such that $G \cong c\cdot\st + \sum_{i = 1}^k b_i \cdot\down^i$
for some integer $k \geq 0$ and sequence $(c, b_1, b_2, \ldots, b_k)$ of non-negative integers.
If $d \defeq= \deg(G) \geq 1$, then the following two statements (i) and (ii) hold.
\begin{enumerate}[(i)]
\item For any $G^L \in G^{\mathcal{L}}$, we have $G^L \leq G - \sum_{i = 1}^d \down^i + \st$.
\item For any $G^R \in G^{\mathcal{R}}$, we have at least one of $G^R = G + \st$ and $G^R \geq G - \down^d$.
\end{enumerate}
\end{lemma}
See Appendix \ref{subsec:proof-uptimal-option} for the proof of Lemma \ref{lem:uptimal-option}.

\subsection{Sequential Compounds}

Now, we introduce the main topic of this paper, sequential compounds of games.
In the sequential compound $G \seq H$ of $G, H \in \short$, the players first play on $G$ until $G$ is over, then play on $H$.
More precisely, a player must make a move on $G$ as long as the player has an option in $G$, and if the player does not, then the player makes a move on $H$.
When either of the two players plays on $H$, the game $G$ is immediately discarded even if the other player still has options in $G$.
The formal definition is as follows.

\begin{definition}[{\cite[Section 6]{SU93}}]
\label{def:seq}
For $G, H \in \tilde{\mathbb{G}}$,
the \emph{sequential compound} $(G \seq H)$ of $G$ and $H$ is defined recursively as
\begin{eqnarray}
\label{eq:r3xyoug24857}
(G \seq H)^{\mathcal{L}}
\defeq{=} \begin{cases}
G^{\mathcal{L}} \seq H &\,\,\text{if}\,\, G^{\mathcal{L}} \neq \emptyset,\\
H^{\mathcal{L}} &\,\,\text{if}\,\, G^{\mathcal{L}} = \emptyset,\\
\end{cases} \quad
(G \seq H)^{\mathcal{R}}
\defeq{=} \begin{cases}
G^{\mathcal{R}} \seq H &\,\,\text{if}\,\, G^{\mathcal{R}} \neq \emptyset,\\
H^{\mathcal{R}} &\,\,\text{if}\,\, G^{\mathcal{R}} = \emptyset,\\
\end{cases}
\end{eqnarray}
where $\mathcal{G} \seq H \defeq{=} \{G \seq H : G \in \mathcal{G}\}$ for $\mathcal{G} \subseteq \short$.
\end{definition}

\begin{remark}
\label{rem:seq}
Definition \ref{def:seq} is from \cite[Section 6]{SU93}.
The definition of sequential compounds in \cite[Definition 3]{Ste07} is slightly different as follows:
\begin{equation}
\label{eq:9znzjudvuld5}
G \seq H \defeq{\cong} \begin{cases}
\left\{ G^{\mathcal{L}} \seq H \relmiddle| G^{\mathcal{R}} \seq H \right\} &\,\,\text{if}\,\, G \not\cong 0,\\
H &\,\,\text{if}\,\, G \cong 0.
\end{cases}
\end{equation}

In this paper, we adopt the definition by \cite[Section 6]{SU93} because it enjoys the good properties that
the outcome class of a game $G$ in mis\`{e}re play can be represented as $o(G \seq \st)$ by the outcome class in normal play.
In this sense, our discussion includes the case of mis\`{e}re play.
Also, the outcome class of a sequential compound by \cite[Section 6]{SU93} is determined by the simple formula in the following Lemma \ref{lem:seq-outcome}.
\end{remark}

\begin{lemma}[{\cite[Theorem 6.1]{SU93}}]
\label{lem:seq-outcome}
For any $G, H \in \short$, we have
\begin{equation*}
o(G \seq H)
= \begin{cases}
\mathscr{L} &\,\,\text{if}\,\, o(H) = \mathscr{L},\\
\mathscr{R} &\,\,\text{if}\,\, o(H) = \mathscr{R},\\
o(G) &\,\,\text{if}\,\, o(H) = \mathscr{P},\\
o(G \seq \st) &\,\,\text{if}\,\, o(H) = \mathscr{N}.\\
\end{cases}
\end{equation*}
\end{lemma}

The sequential compound satisfies the following basic properties, such as the associative law and the existence of an identity.

\begin{lemma}
\label{lem:seq}
The following statements (i)--(iii) hold.
\begin{enumerate}[(i)]
\item For any $G \in \tilde{\mathbb{G}}$, we have $G \seq 0 \cong 0 \seq G \cong G$.
\item For any $G, H, J \in \tilde{\mathbb{G}}$, we have $(G\seq H) \seq J \cong G \seq (H \seq J)$.
\item For any $G, H \in \tilde{\mathbb{G}}$, we have $-(G \seq H) \cong (-G) \seq (-H)$.
\end{enumerate}
\end{lemma}

See Appendix \ref{subsec:proof-seq} for the proof of Lemma \ref{lem:seq}.
\footnote{Lemma \ref{lem:seq} (i) and (ii) are stated by \cite{SU93}.
However, because no proof is found in \cite{SU93}, we give an explicit proof here.}

By the associativity, we can write the sequential compound of multiple games $G_1, G_2, \ldots, G_n$ unambiguously as
\begin{equation}
\label{eq:gwnvh7zozx9g}
G_1 \to G_2 \to \cdots \to G_n.
\end{equation}
In (\ref{eq:gwnvh7zozx9g}), a player makes a move on the leftmost component $G_i$ such that the player has an option,
and then the components $G_1, G_2, \ldots, G_{i-1}$ are discarded resulting in $G'_i \to G_{i+1} \to \cdots \to G_n$, where $G'_i$ is the resulting game by the move on $G_i$.

\section{Main Results}
\label{sec:main}

In this section, we give the game value of any given sequential compound of integers and stars, that is,
\begin{equation}
\label{eq:t9wj7zfk2nku}
G \defeq{\cong} G_1 \seq G_2 \seq \cdots \seq G_n, \,\,\text{where}\,\, G_1, G_2, \ldots, G_n \in \mathbb{Z} \cup \{\st\}
\end{equation}
as our main result.
More precisely, we prove that $G$ can be written as $G = c\st + \sum_{i = 1}^k b_i \up^k$ for some integers $c \geq 0, k \geq 0$ and a sequence $(b_1, b_2, \ldots, b_k)$ of integers, and we give the specific values of $c, k$, and $(b_1, b_2, \ldots, b_k)$.

To do this, we first show some general properties of sequential compounds, parts of the main results in itself, in Subsection \ref{sec:main-general}.
In Subsection \ref{sec:main-int}, we consider the game value of (\ref{eq:t9wj7zfk2nku}) in the special case where $G_1, G_2, \ldots, G_n$ are integers.
Then we show the game value for the general case containing stars in Subsection \ref{sec:main-int-star}.

\subsection{General Properties of Sequential Compounds}
\label{sec:main-general}

In this subsection, we state three theorems that give general properties of sequential compounds.

The first theorem shows that if $H$ is a non-zero dicotic game, then $G \seq H$ is also a non-zero dicotic game.

\begin{theorem}
\label{lem:dicot}
For any $G \in \tilde{\mathbb{G}}$ and $H \in \dicot \setminus \{0\}$, we have $G \seq H \in \dicot \setminus \{0\}$.
\end{theorem}

\begin{proof}[Proof of Theorem \ref{lem:dicot}]
We prove by induction on $\birth(G)$.
For the base case $G \cong 0$, we have
\begin{equation*}
G \seq H \cong 0 \seq H \cong H \in \dicot \setminus \{0\}.
\end{equation*}

We next consider the induction hypothesis for $G \not \cong 0$.
Here, we show $\emptyset \neq (G \seq H)^{\mathcal{L}} \subseteq \dicot$, and then it follows that $\emptyset \neq (G \seq H)^{\mathcal{R}} \subseteq \dicot$ by the symmetric argument.
Since 
\begin{eqnarray*}
(G \seq H)^{\mathcal{L}}
= \begin{cases}
G^{\mathcal{L}} \seq H &\,\,\text{if}\,\, G^{\mathcal{L}} \neq \emptyset,\\
H^{\mathcal{L}} &\,\,\text{if}\,\, G^{\mathcal{L}} = \emptyset
\end{cases}
\end{eqnarray*}
by (\ref{eq:r3xyoug24857}), it suffices to show that $\emptyset \neq G^{\mathcal{L}} \seq H \subseteq \dicot$ in the case $G^{\mathcal{L}} \neq \emptyset$, and
 $\emptyset \neq H^{\mathcal{L}} \subseteq \dicot$ in the case $G^{\mathcal{L}} = \emptyset$ as follows.
\begin{itemize}
\item The case $G^{\mathcal{L}} \neq \emptyset$:
We have
\begin{equation}
\emptyset \eqlab{A}\neq G^{\mathcal{L}} \seq H = \left\{  G^L \seq H : G^L \in G^{\mathcal{L}} \right\} \eqlab{B}\subseteq \dicot,
\end{equation}
where
(A) follows from $G^{\mathcal{L}} \neq \emptyset$,
and (B) follows by applying the induction hypothesis to every $G^L \seq H$.

\item The case $G^{\mathcal{L}} = \emptyset$: We have $\emptyset \neq H^{\mathcal{L}} \subseteq \dicot$ directly from the assumption $H \in \dicot \setminus \{0\}$.
\end{itemize}
\end{proof}

Next, we present the second theorem.
In sequential compounds, it is not guaranteed in general that replacing a component with another with an equal game value does not change the game value in contrast to the case of disjunctive compounds (cf. Proposition \ref{prop:order} (iii) and Corollary \ref{cor:replace}).
Namely, even if $G = G'$ (resp.~$H = H'$), it does not necessarily hold that $G \seq H = G' \seq H$ (resp.~$G \seq H = G \seq H'$).
However, in the case $G$ is dicotic, the condition $H = H'$ guarantees $G \seq H = G \seq H'$, that is, we can replace $H$ with $H'$ without changing the game value.
More generally, the following Theorem \ref{lem:order-preserve} holds.

\begin{theorem}
\label{lem:order-preserve}
For any $G \in \dicot$ and $H, H' \in \tilde{\mathbb{G}}$, we have
\begin{equation}
\label{eq:95f7071iydnv}
H \geq H' \implies (G \seq H) \geq (G \seq H').
\end{equation}
\end{theorem}

\begin{proof}[Proof of Theorem \ref{lem:order-preserve}]
We assume $H - H' \geq 0$ and show
$J \defeq{\cong} (G \seq H) - (G \seq H') \geq 0$
by induction on $\birth(G)$.

For the base case $G \cong 0$, we have $(G \seq H) - (G \seq H') \cong (0 \seq H) - (0 \seq H') \cong H - H' \geq 0$ as desired.

We consider the induction step for $G \not\cong 0$ (i.e., $G \in \dicot \setminus \{0\}$).
To prove $J \geq 0$, it suffices to show that for any $J^R \in J^{\mathcal{R}}$, there exists $J^{RL} \in (J^R)^{\mathcal{L}}$ such that $J^{RL} \geq 0$.
Because $G^{\mathcal{L}} \neq \emptyset$ and $G^{\mathcal{R}} \neq \emptyset$ by $G \in \dicot \setminus \{0\}$,
any $J^R \in J^{\mathcal{R}}$ is in one of the following two forms (i) and (ii):
\begin{enumerate}[(i)]
\item $J^R \cong (G^R \seq H) - (G \seq H')$ for some $G^R \in G^{\mathcal{R}}$,
\item $J^R \cong (G \seq H) - (G^L \seq H')$ for some $G^L \in G^{\mathcal{L}}$.
\end{enumerate}

In the case (i), the Right option $J^R$ has the Left option
\begin{equation*}
(G^R \seq H) - (G^R \seq H') \eqlab{A}{\geq} 0,
\end{equation*}
where (A) follows from the induction hypothesis.

In the case (ii), the Right option $J^R$ has the Left option
\begin{equation*}
(G^L \seq H) - (G^L \seq H') \eqlab{A}{\geq} 0,
\end{equation*}
where (A) follows from the induction hypothesis.
\end{proof}

Since the symmetric argument holds when replacing $\geq$ in (\ref{eq:95f7071iydnv}) with $\leq$,
we obtain the following corollary.

\begin{corollary}
\label{cor:order-preserve}
For any $G \in \dicot$ and $H, H' \in \tilde{\mathbb{G}}$, we have
\begin{equation*}
H = H' \implies (G \seq H) = (G \seq H').
\end{equation*}
\end{corollary}

\begin{remark}
The converse of (\ref{eq:95f7071iydnv}) does not necessarily hold because of the following counterexample:
$G \defeq{\cong} \st, H \defeq{\cong} \st, H' \defeq{\cong} \{0, \st \mid 0, \st \}$.
\end{remark}

Next, we describe the third theorem.
In the aspect of the Number Avoidance Theorem mentioned above Proposition \ref{prop:number-transition},
it is expected that in the disjunctive sum $G + X$ of a non-number $G$ and a number $X$,
both players continue to make moves on $G$ at first until $G$ is over, and then play on $X$.
Namely, the play of $G + X$ proceeds like as $G \seq X$, and thus $G+X = G \seq X$ holds.
This expectation holds true for a dicotic $G$ in a more general form as the following theorem.

\begin{theorem}
\label{thm:num-trans}
For any $G \in \dicot$ and $H, X \in \tilde{\mathbb{G}}$, if $G \seq H \in \dicot$ and $X = x$ for some $x \in \mathbb{D}$,
then
$G \seq (X + H) = X + (G \seq H)$.
\end{theorem}

\begin{proof}[Proof of Theorem \ref{thm:num-trans}]
We prove by induction on $\birth(G)$.

For the base case $G \cong 0$, we have $G \seq (X + H) \cong 0 \seq (X + H) \cong X + H \cong  X + (0 \seq H) \cong X + (G \seq H)$.

We consider the induction step for $G \not\cong 0$ (i.e., $G \in \dicot \setminus \{0\}$) dividing into the following two cases: the case $G \seq H \neq 0$ and the case $G \seq H = 0$.
\begin{itemize}
\item The case $G \seq H \neq 0$:
We have
\begin{eqnarray*}
G \seq (X+H)
&\eqlab{A}{\cong}& \left\{G^{\mathcal{L}} \seq (X+H) \relmiddle| G^{\mathcal{R}} \seq (X+H)\right\}\\
&{\cong}& \left\{ \{G^{L} \seq (X+H) : G^L \in G^{\mathcal{L}}\} \relmiddle| \{G^{R} \seq (X+H) : G^R \in G^{\mathcal{R}}\} \right\}\\
&\eqlab{B}{=}& \left\{ \{X + (G^{L} \seq H) : G^L \in G^{\mathcal{L}}\} \relmiddle| \{ X + (G^{R} \seq H) : G^R \in G^{\mathcal{R}}\} \right\}\\
&\cong& \left\{X + (G^{\mathcal{L}} \seq H) \relmiddle| X + (G^{\mathcal{R}} \seq H)\right\}\\
&\eqlab{C}\cong& \left\{X + (G \seq H)^{\mathcal{L}} \relmiddle| X + (G \seq H)^{\mathcal{R}} \right\}\\
&\eqlab{D}{=}& X + \left\{(G \seq H)^{\mathcal{L}} \relmiddle| (G \seq H)^{\mathcal{R}} \right\}\\
&\cong& X + (G \seq H),
\end{eqnarray*}
where
(A) follows from (\ref{eq:r3xyoug24857}) since $G^{\mathcal{L}}  \neq \emptyset$ and $G^{\mathcal{R}} \neq \emptyset$ are guaranteed by $G \in \dicot \setminus \{0\}$,
(B) follows from the induction hypothesis since $G^L, G^R, (G^L \seq H)$, and $(G^R \seq H)$ are dicotic for every $G^L \in G^{\mathcal{L}}$ and $G^R \in G^{\mathcal{R}}$
since $G$ and $(G \seq H)$ are dicotic,
(C) follows from (\ref{eq:r3xyoug24857}) since $G^{\mathcal{L}}  \neq \emptyset$ and $G^{\mathcal{R}} \neq \emptyset$ are guaranteed by $G \in \dicot \setminus \{0\}$,
and (D) follows from Proposition \ref{prop:number-transition} because $G \seq H$ is not a number as follows:
for any positive $y \in \dyadic$, we have $-y < G \seq H < y$ by $G \seq H \in \dicot$ and Proposition \ref{prop:lawnmower};
hence, by the assumption $G \seq H \neq 0$, there is no $y \in \dyadic$ such that $G \seq H = y$.

\item The case $G \seq H = 0$:
It suffices to show $(G \seq (X+H)) - x = 0$ because it implies
\begin{equation*}
(G \seq (X+H)) - (X + (G \seq H)) = (G \seq (X+H)) - (x + 0) \cong  (G \seq (X+H)) - x = x - x = 0.
\end{equation*}
We prove only $(G \seq (X+H)) - x \leq 0$ by showing that Left loses in $(G \seq (X+H)) - x$ playing first.
Then the opposite inequality $(G \seq (x+H)) - x \geq 0$ follows from the symmetric argument.

First, any Left's move on the component $x$ makes Left lose because
any Left option $(G \seq (X+H)) - x^R$ has the Right option
\begin{equation*}
(G^R \seq (X+H)) - x^R
\eqlab{A}{=} X + (G^R \seq H) - x^R
= (G^R \seq H) + (x - x^R)
\eqlab{B}{<} 0,
\end{equation*}
where
(A) follows from the induction hypothesis since $G^R, (G^R \seq H) \in \dicot$ by $G, (G \seq H) \in \dicot$, 
(B) follows because $x - x^R$ is a negative number by (\ref{eq:d21j3vzuqcup}), and thus $x - x^R < G^R \seq H$ by Proposition \ref{prop:lawnmower} and $G^R \seq H \in \dicot$.

Hence, we consider Left's moves on the component $G \seq (X+H)$, which results in $(G^L \seq (X+H)) - x$ for some $G^L \in G^{\mathcal{L}}$.
We consider the following two cases separately: the case $G^L \cong 0$ and the case $G^L \in \dicot \setminus \{0\}$.

\begin{itemize}
\item The case $G^L \cong 0$:
By $G \seq H = 0$, the Left option $G^L \seq H \cong H$ of $G \seq H$ has a Right option $H^R  \leq 0$.
Thus, the game $(G^L \seq (X+H)) - x \cong (X+H) - x$ has the Right option
\begin{equation*}
(X+H^R) - x = x+H^R - x = H^R \leq 0.
\end{equation*}

\item The case $G^L \in \dicot \setminus \{0\}$:
By $G \seq H = 0$, the Left option $G^L \seq H$ of $G \seq H$ has a Right option $G^{LR} \seq H \leq 0$.
Thus, the game $(G^L \seq (X+H)) - x$ has the Right option
\begin{equation*}
(G^{LR} \seq (X+H)) - x
\eqlab{A}{=} X + (G^{LR} \seq H) - x
= x + (G^{LR} \seq H) - x
= G^{LR} \seq H
\leq 0,
\end{equation*}
where
(A) follows from the induction hypothesis since $G^{LR}, (G^{LR} \seq H) \in \dicot$ by $G, (G \seq H) \in \dicot$.
\end{itemize}
\end{itemize}
\end{proof}

Applying Theorem \ref{thm:num-trans} with $H \defeq{\cong} 0$ yields the following corollary.

\begin{corollary}
\label{cor:num-trans}
For any $G \in \dicot$ and $x \in \mathbb{D}$, we have $G \seq x = G + x$.
\end{corollary}

\subsection{Sequential Compounds of Integers}
\label{sec:main-int}

In this subsection, we consider the game value of (\ref{eq:t9wj7zfk2nku}) in the special case where $G_i \in \mathbb{Z}$ for all $i = 1, 2, \ldots, n$.

First, we confirm the following Theorem \ref{thm:int} that
a positive (resp.~negative) integer can be decomposed into the sequential compound of $n$ copies of $1$ (resp.~$-1$).

\begin{theorem}
\label{thm:int}
For any integer $n \geq 1$, we have
\begin{equation*}
\underbrace{1 \seq 1 \seq \cdots \seq 1}_{n} \cong n.
\end{equation*}
\end{theorem}

\begin{proof}[Proof of Theorem \ref{thm:int}]
We prove by induction on $n$.
The base case $n = 1$ is trivial.
For the induction step for $n \geq 2$, we have
\begin{equation*}
\underbrace{1 \seq 1 \seq \cdots \seq 1}_{n}
\cong 1 \seq \underbrace{(1 \seq \cdots \seq 1)}_{n-1}
\eqlab{A}{\cong} 1 \seq \llbracket n-1 \rrbracket
\cong \{0|\}  \seq \llbracket n-1 \rrbracket
\eqlab{B}{\cong} \{0 \seq \llbracket n-1 \rrbracket\mid\}
\cong \{\llbracket n-1 \rrbracket\mid\}
\cong n,
\end{equation*}
where (A) follows from the induction hypothesis,
and (B) follows from (\ref{eq:r3xyoug24857}).
\end{proof}

By Theorem \ref{thm:int} and the associativity of sequential compounds,
finding the game value of (\ref{eq:t9wj7zfk2nku}) can be reduced to the case of
\begin{equation}
\label{eq:9h59ocz03gr4}
G_1 \seq G_2 \seq \cdots \seq G_n, \,\,\text{where}\,\, G_1, G_2, \ldots, G_n \in \{1, -1, \st\},
\end{equation}
by replacing all integers in (\ref{eq:t9wj7zfk2nku}) with the sequential compounds of $1$ and $-1$.
For example, 
\begin{eqnarray*}
\lefteqn{(-3) \seq \st \seq 2 \seq 3 \seq (-2) \seq \st \seq \st \seq 1}\\
&\cong& \left((-1) \seq (-1) \seq (-1)\right) \seq \st \seq \left(1 \seq 1\right) \seq \left(1 \seq 1 \seq 1\right) \seq
\left((-1) \seq (-1)\right) \seq \st \seq \st \seq 1\\
&\cong& (-1) \seq (-1) \seq (-1) \seq \st \seq 1 \seq 1 \seq 1 \seq 1 \seq 1 \seq (-1) \seq (-1) \seq \st \seq \st \seq 1.
\end{eqnarray*}

By Theorem \ref{thm:int}, the game value of (\ref{eq:9h59ocz03gr4}) is easily obtained in the case where 
$G_i \cong 1$ for all $i = 1, 2, \ldots, n$ and the case where $G_i \cong -1$ for all $i = 1, 2, \ldots, n$.
Hence, we consider the case $\{G_1, G_2, \ldots, G_n\} = \{1, -1\}$.
Without loss of generality, we may assume $G_n \cong 1$ by negating the whole sequential compound if necessary, and thus it suffices to consider only the form
\begin{equation*}
G_1 \seq G_2 \seq \cdots \seq G_n \seq (-1) \seq 1 \seq \underbrace{1 \seq \cdots \seq 1}_m, \,\,\text{where}\,\, G_1, G_2, \ldots, G_n \in \{1, -1\},
\end{equation*}
for an integer $m \geq 0$.
Then the desired game value is given as follows.

\begin{theorem}
\label{thm:int2}
For any integers $m \geq 0, n \geq 0$ and sequence $(a_0, a_1, a_2, \ldots, a_n)$ of integers such that $a_i \in \{1, -1\}$ for all $i$, $a_0 = 1$ and $a_1 = -1$,
we have
\begin{equation}
\label{eq:oi8vpbgnj8m6}
a_n \seq \cdots \seq a_1 \seq a_0 \seq m
= m + \sum_{i = 0}^n \frac{a_i}{2^i}.
\end{equation}
\end{theorem}

\begin{proof}[Proof of Theorem \ref{thm:int2}]
We prove by induction on $n$.
For the base case $n = 0$, we have
\begin{equation*}
a_0 \seq m
\cong 1 \seq m
\eqlab{A}\cong \llbracket m+1 \rrbracket
= m+1
= m + \sum_{i = 0}^n \frac{a_i}{2^i},
\end{equation*}
where (A) follows from Theorem \ref{thm:int}.
We consider the induction step for $n \geq 1$ dividing into the following two cases: the case $a_n = 1$ and the case $a_n = -1$.
\begin{itemize}
\item The case $a_n = 1$:
By $a_1 = -1$, there exists an integer $1 \leq p \leq n-1$ such that
\begin{equation}
\label{eq:yre1jqaxsjr1}
a_n = a_{n-1} = \cdots = a_{p+1} = 1, \quad a_p = -1.
\end{equation}
We have
\begin{eqnarray*}
\lefteqn{a_n \seq \cdots \seq a_1 \seq a_0 \seq m}\\
&\eqlab{A}{\cong}& \left\{a_{n-1} \seq \cdots \seq a_1 \seq a_0 \seq m \relmiddle|
a_{p-1} \seq \cdots \seq a_1 \seq a_0 \seq m \right\}\\
&\eqlab{B}{=}& \left\{m + \sum_{i = 0}^{n-1} \frac{a_i}{2^i} \relmiddle| m + \sum_{i = 0}^{p-1} \frac{a_i}{2^i} \right\}\\
&=& \left\{m + \frac{1}{2^n}\sum_{i = 0}^{n-1} 2^{n-i}a_i  \relmiddle| m + \frac{1}{2^n}\sum_{i = 0}^{p-1} 2^{n-i}a_i  \right\}\\
&\eqlab{C}{=}& \left\{m + \frac{1}{2^n}\left(\sum_{i = 0}^{p-1} 2^{n-i}a_i   - 2^{n-p} + \sum_{i = p+1}^{n-1} 2^{n-i} \right) \relmiddle| m + \frac{1}{2^n} \sum_{i = 0}^{p-1} 2^{n-i}a_i   \right\}\\
&=& \left\{m + \frac{1}{2^n}\left(\sum_{i = 0}^{p-1} 2^{n-i}a_i   - 2\right) \relmiddle| m + \frac{1}{2^n}\left(\sum_{i = 0}^{p-1} 2^{n-i}a_i \right) \right\}\\
&\eqlab{D}{=}& m + \frac{1}{2^n}\left(\sum_{i = 0}^{p-1} 2^{n-i}a_i  - 1 \right)\\
&=& m + \frac{1}{2^n}\left(\sum_{i = 0}^{p-1} 2^{n-i}a_i - 2^{n-p} + \sum_{i = p+1}^{n} 2^{n-i} \right)\\
&\eqlab{E}{=}& m + \sum_{i = 0}^n \frac{a_i}{2^i}
\end{eqnarray*}
as desired, where
(A) follows from (\ref{eq:yre1jqaxsjr1}) and the definition of sequential compounds,
(B) follows from the induction hypothesis,
(C) follows from (\ref{eq:yre1jqaxsjr1}),
(D) follows from (\ref{eq:g0h3zqunv24n}),
and (E) follows from (\ref{eq:yre1jqaxsjr1}).

\item The case $a_n = -1$:
By $a_0 = 1$, there exists an integer $0 \leq p \leq n-1$ such that
\begin{equation*}
a_n = a_{n-1} = \cdots = a_{p+1} = -1, \quad a_p = 1.
\end{equation*}
By (\ref{eq:yre1jqaxsjr1}) and the definition of sequential compounds, we have
\begin{eqnarray*}
a_n \seq \cdots \seq a_1 \seq a_0 \seq m
\eqlab{A}{\cong} \left\{ a_{p-1} \seq \cdots \seq a_1 \seq a_0 \seq m
\relmiddle|  a_{n-1} \seq \cdots \seq a_1 \seq a_0 \seq m \right\},
\end{eqnarray*}
where $a_{p-1} \seq \cdots \seq a_1 \seq a_0 \seq m$ is regarded as $m$ if $p = 0$.
Then the desired result follows from the symmetric discussion to the case $a_n = 1$.
\end{itemize}
\end{proof}

\begin{remark}
For the sequential compound of (\ref{eq:oi8vpbgnj8m6}), we consider the following ordinal sum
\begin{equation}
\label{eq:t1qu4m05lzyo}
G \defeq\cong m \osum a_0 \osum a_1 \osum \cdots \osum a_{n-1} \osum a_n,
\end{equation}
in which a player plays on exactly one component, and then the components to the right of the played one are immediately discarded.
The well-known result
\begin{equation}
\label{eq:1pa48vr656z7}
G = m + \sum_{i = 0}^n \frac{a_i}{2^i},
\end{equation}
which has an identical right-hand side to (\ref{eq:oi8vpbgnj8m6}),
can be found in \cite{ANW19, BCG18, Con00, HG16, Sie13}, where the ordinal sum is represented as a position of \emph{Hackenbush}.
Note that the sequential compound (\ref{eq:oi8vpbgnj8m6}) is obtained by imposing the following constraint to  (\ref{eq:t1qu4m05lzyo}):
a player must play on the rightmost component in which the player has an option.
On the other hand, \cite[Theorem 31]{Wah14} shows that a player does not need to play the components other than
the rightmost one such that the player has an option in (\ref{eq:t1qu4m05lzyo}) as a more general statement. 
Therefore, it is expected that imposing the constraint to (\ref{eq:t1qu4m05lzyo}) does not affect the game value, and the formula (\ref{eq:1pa48vr656z7}) can be diverted to sequential compounds.
This discussion is consistent with Theorem \ref{thm:int2}.
\end{remark}

\subsection{Sequential Compounds of Integers and Stars}
\label{sec:main-int-star}

In this subsection, we show the game value of (\ref{eq:9h59ocz03gr4}) in the general case $G_1, G_2, \ldots, G_n \in \{1, -1, \st\}$.

By focusing on the rightmost star, (\ref{eq:9h59ocz03gr4}) can be seen as
\begin{equation}
\label{eq:rvozr2sroqib}
G_1 \seq G_2 \seq \cdots \seq G_{p-1} \seq \st \seq G_{p+1} \seq G_{p+2} \seq \cdots \seq G_{n}
\end{equation}
for some $1 \leq p \leq n$,
where $G_1, G_2, \ldots, G_{p-1} \in \{1, -1, \st\}$ and $G_{p+1}, G_{p+2}, \ldots, G_n \in \{1, -1\}$.
Then we have
\begin{equation}
\label{eq:utq0wz6zqy06}
\left( G_1 \seq G_2 \seq \cdots \seq G_{p-1} \seq \st  \right) \in \dicot
\end{equation}
by Theorem \ref{lem:dicot} since $\st \in \dicot \setminus \{0\}$.
Also, it holds that
\begin{equation}
\label{eq:jyskxwxnmemc}
G_{p+1} \seq G_{p+2} \seq \cdots \seq G_{n} = x
\end{equation}
for some $x \in \dyadic$ by Theorems \ref{thm:int} and \ref{thm:int2} since $G_{p+1}, G_{p+2}, \ldots, G_n \in \{1, -1\}$.
Hence, we have
\begin{eqnarray*}
\lefteqn{G_1 \seq G_2 \seq \cdots \seq G_{p-1} \seq \st \seq G_{p+1} \seq G_{p+2} \seq \cdots \seq G_{n}} \nonumber\\
&\eqlab{A}{=}& G_1 \seq G_2 \seq \cdots \seq G_{p-1} \seq \st \seq x \\
&\eqlab{B}{=}& \left(G_1 \seq G_2 \seq \cdots \seq G_{p-1} \seq \st\right) + x,
\end{eqnarray*}
where
(A) follows from (\ref{eq:utq0wz6zqy06}), (\ref{eq:jyskxwxnmemc}), and Corollary \ref{cor:order-preserve},
and (B) follows (\ref{eq:utq0wz6zqy06}), (\ref{eq:jyskxwxnmemc}),  and Corollary \ref{cor:num-trans}.
Therefore, it suffices to find the game value of $G_1 \seq G_2 \seq \cdots \seq G_{p-1} \seq \st$,
that is, we may assume that the rightmost component of the given sequential compound is a star without loss of generality.

Further, if there are three or more stars at the end of the given sequential compound, we may remove two of them by (ii) of the following theorem.
\begin{theorem}
\label{thm:star}
For any $G \in \tilde{\mathbb{G}}$, the following statements (i) and (ii) hold.
\begin{enumerate}[(i)]
\item If $G \in \dicot$, then $G \seq \st \seq \st= G$.
\item $G \seq \st \seq \st \seq \st = G \seq \st$.
\end{enumerate}
\end{theorem}

\begin{proof}[Proof of Theorem \ref{thm:star}]
(Proof of (i)) 
We have
\begin{equation*}
G \seq (\st \seq \st)
\eqlab{A}{=} G \seq 0 \cong G,
\end{equation*}
where (A) follows from Corollary \ref{cor:order-preserve} since $G \in \dicot$ and $\st \seq \st = 0$.

(Proof of (ii))
We have $G \seq \st \seq \st \seq \st = (G \seq \st) \seq \st \seq \st  = G \seq \st$ directly from (i) of this lemma since $G \seq \st \in \dicot$ by Theorem \ref{lem:dicot}.
\end{proof}

Since we can remove two stars by applying Theorem \ref{thm:star} as long as there are three or more stars at the end, 
we may assume that there are at most two stars at the end of the given sequential compound without loss of generality.
Namely, it suffices to consider only the following two forms:
\begin{eqnarray}
&&G_1 \seq G_2 \seq \cdots \seq G_n \seq \st, \label{eq:8qqobw8bib4c}\\
&&G_1 \seq G_2 \seq \cdots \seq G_n \seq \st \seq \st, \label{eq:0lbxvnb8pohv}
\end{eqnarray}
where $n \geq 0$ and $G_n \in \{1, -1\}$.

The case $n = 0$ is straightforward since (\ref{eq:8qqobw8bib4c}) is $\st$ and (\ref{eq:0lbxvnb8pohv}) is $\st \seq \st = 0$ then.
Thus, by negating the whole sequential compound if necessary, we may limit to the following two forms:
\begin{eqnarray}
&&G_1 \seq G_2 \seq \cdots \seq G_n \seq 1 \seq \st, \label{eq:slcsc43u2u2z}\\
&&G_1 \seq G_2 \seq \cdots \seq G_n \seq (-1) \seq \st \seq \st, \label{eq:2vmcdrxoww2z}
\end{eqnarray}
where $n \geq 0$.

By separating the position of each star,  (\ref{eq:slcsc43u2u2z}) and (\ref{eq:2vmcdrxoww2z}) can be decomposed into
\begin{equation}
\label{eq:z09f8dujvi74}
\left(G^{(m)}_{l^{(m)}} \seq \cdots \seq G^{(m)}_1 \seq \st\right)
\seq \left(G^{(m-1)}_{l^{(m-1)}}\seq \cdots \seq G^{(m-1)}_1 \seq \st\right)
\seq \cdots
\seq \left(G^{(1)}_{l^{(1)}}  \seq \cdots  \seq G^{(1)}_1\seq \st\right)
\end{equation}
for some integers $m \geq 1$, $l^{(i)} \geq 0$, and $G^{(i)}_j \in \{1, -1\}$.
Further, each term $\left(G^{(i)}_{l^{(i)}}\seq \cdots \seq G^{(i)}_1 \seq \st\right)$ is in the form
\begin{equation}
\label{eq:zva05m56z6rt}
\underbrace{1 \seq 1 \seq \cdots \seq 1}_{a_n} \seq (-1) \seq \underbrace{1 \seq 1 \seq \cdots \seq 1}_{a_{n-1}} \seq (-1) \seq
\cdots \seq \underbrace{1 \seq 1 \seq \cdots \seq 1}_{a_1} \seq (-1) \seq \underbrace{1 \seq 1 \seq \cdots \seq 1}_{a_0} \seq \st
\end{equation}
for some $n \geq 0$ and sequence $(a_0, a_1, \ldots, a_n)$ of non-negative integers.
Therefore,  (\ref{eq:slcsc43u2u2z}) and (\ref{eq:2vmcdrxoww2z}) can be decomposed into a sequential compound of games in the form (\ref{eq:zva05m56z6rt}).
We define a notation $[a_n, a_{n-1}, \ldots, a_2, a_1, a_0]$ for this as follows.
\begin{definition}
\label{def:block}
For any integer $n \geq 0$ and sequence $(a_0, a_1, \ldots, a_n)$ of non-negative integers, we define
\begin{equation*}
[a_n, a_{n-1}, \ldots, a_1, a_0]  \defeq{\cong} a_n \seq (-1) \seq a_{n-1} \seq (-1) \seq \cdots \seq (-1) \seq a_1 \seq (-1) \seq  a_0 \seq \st.
\end{equation*}
\end{definition}

By this notation, (\ref{eq:slcsc43u2u2z}) can be decomposed as follows.
\begin{equation}
\label{eq:cmw7ygqcgqi2}
\left [a^{(m)}_{n^{(m)}}, a^{(m)}_{n^{(m)}-1}, \ldots, a^{(m)}_0 \right]
\seq \left[a^{(m-1)}_{n^{(m-1)}}, a^{(m-1)}_{n^{(m-1)}-1}, \ldots, a^{(m-1)}_0 \right] 
\seq \cdots \seq \left[a^{(1)}_{n^{(1)}}, a^{(1)}_{n^{(1)}-1}, \ldots, a^{(1)}_0 \right],
\,\,\text{where}\,\, a^{(1)}_0 \geq 1
\end{equation}
for some $m \geq 1$.
Similarly, (\ref{eq:2vmcdrxoww2z}) can be decomposed as
\begin{eqnarray}
\label{eq:mh9njmbt23ou}
\left [a^{(m)}_{n^{(m)}}, a^{(m)}_{n^{(m)}-1}, \ldots, a^{(m)}_0 \right]
\seq \left[a^{(m-1)}_{n^{(m-1)}}, a^{(m-1)}_{n^{(m-1)}-1}, \ldots, a^{(m-1)}_0 \right] 
\seq \cdots \seq \left[a^{(1)}_{n^{(1)}}, a^{(1)}_{n^{(1)}-1}, \ldots, a^{(1)}_0 \right]
\seq \st, \nonumber\\
\text{where}\,\, n^{(1)} \geq 1 \,\,\text{and}\,\, a^{(1)}_0 = 0
\end{eqnarray}
for some $m \geq 1$.
The last component $\st$ in (\ref{eq:mh9njmbt23ou}) can be replaced with $[0]$; however, we leave it as $\st$ for later discussion.

\begin{example}
A game
\begin{equation*}
G \defeq{\cong} (-1) \seq 1 \seq 1 \seq (-1) \seq 1 \seq \st \seq \st \seq 1 \seq 1 \seq (-1) \seq 1 \seq 1 \seq 1 \seq (-1) \seq \st \seq \st
\end{equation*}
in the form (\ref{eq:2vmcdrxoww2z}) is rewritten as
\begin{eqnarray*}
G &\cong& \left((-1) \seq 1 \seq 1 \seq (-1) \seq 1 \seq \st\right) \seq \left(\st\right) \seq \left(1 \seq 1 \seq (-1) \seq 1 \seq 1 \seq 1 \seq (-1) \seq \st\right) \seq \st \\
&\cong& \left(0 \seq (-1) \seq 2 \seq (-1) \seq 1 \seq \st\right) \seq \left(0 \seq \st\right) \seq \left(2 \seq (-1) \seq 3 \seq (-1) \seq 0 \seq \st\right) \seq \st \\
&\cong& [0, 2, 1] \seq [0] \seq [2, 3, 0] \seq \st
\end{eqnarray*}
in the form (\ref{eq:mh9njmbt23ou}).
\end{example}

The following Theorem \ref{thm:int-star} is a key theorem to obtain the game values of (\ref{eq:cmw7ygqcgqi2}) and (\ref{eq:mh9njmbt23ou}).
From now on, for $n \in \mathbb{Z}$ and $G \in \short$, we write $n \cdot G$ as $nG$ for simplicity\footnote{In combinatorial game theory, for $n \in \mathbb{Z}$, the notation $n\st$  usually stands for $n + \st$,
while we use $n\st$ to express $n \cdot \st$ here.
Note that $n\st \cong n \cdot \st$ is equal to $0$ for even $n$, and $\st$ for odd $n$ in our notation.}.

\begin{theorem}
\label{thm:int-star}
Choose an integer $n \geq 0$, a sequence $(a_0, a_1, \ldots, a_n)$ of non-negative integers, and $J \in \short$ arbitrarily.
Assume that one of the following conditions (a)--(c) holds.
\begin{enumerate}[(a)]
\item $J\cong 0, a_0 \geq 1$.
\item $J \cong \st, n \geq 1, a_0 = 0$.
\item $J < 0$ and
\begin{equation}
\label{eq:dnqwaieisfhc}
J \cong c \st + \sum_{i = 1}^k b_i \down^i
\end{equation}
for some integer $k \geq1$ and $(c, b_1, b_2, \ldots, b_k)$ of non-negative integers such that
\begin{equation}
\label{eq:3dj9mrwnisxe}
b_1 \geq b_2 \geq \cdots \geq b_k \geq 1.
\end{equation}
\end{enumerate}
Then we have
\begin{equation*}
[a_n, a_{n-1}, \ldots, a_0] \seq J = \langle a_n, a_{n-1}, \ldots, a_0\rangle_{J},
\end{equation*}
where
\begin{align}
\label{eq:xhzxoe0bbtgq}
\langle a_n, a_{n-1}, \ldots, a_0\rangle_{J}
\defeq{\cong} \begin{cases}
\left(1 + \displaystyle\sum_{i = 0}^n a_i \right) \st + \displaystyle\sum_{i = 0}^n \left(1 + \displaystyle\sum_{j = i}^n a_j \right) \down^{i+1} - \down &\,\,\text{if}\,\, J \cong 0,\\
\left(\displaystyle\sum_{i = 1}^n a_i \right) \st + \displaystyle\sum_{i = 1}^n \left(1 + \displaystyle\sum_{j = i}^n a_j \right) \down^i &\,\,\text{if}\,\, J \cong \st,\\
 J + \left(1+\displaystyle\sum_{i = 0}^n a_i \right) \st + \displaystyle\sum_{i = 1}^k \left(1+ \displaystyle\sum_{j = 0}^n a_j \right)  \down^i + \displaystyle\sum_{i = 1}^n \left(1 + \displaystyle\sum_{j = i}^n a_j \right) \down^{k+i} &\,\,\text{if}\,\, J < 0.\\
\end{cases}
\end{align}
\end{theorem}

The proof of Theorem \ref{thm:int-star} is deferred in Appendix \ref{sec:proof-int-star}, which relies on the following lemma.

\begin{lemma}
\label{lem:seq-neg}
For any integer $n \geq 0$, sequence $(a_0, a_1, \ldots, a_n)$ of non-negative integers, and $J \in \short$, if the one of the conditions (a)--(c) in Theorem \ref{thm:int-star} holds, then $[a_n, a_{n-1}, \ldots, a_0] \seq J < 0$.
\end{lemma}
See Appendix \ref{subsec:proof-seq-neg} for the proof of Lemma \ref{lem:seq-neg}.

Now, we consider the game value of (\ref{eq:cmw7ygqcgqi2}).
We obtain
\begin{equation}
\label{eq:b17wbdx2owaa}
\left[a^{(1)}_{n^{(1)}}, a^{(1)}_{n^{(1)}-1}, \ldots, a^{(1)}_0 \right]
= \left(1 + \displaystyle\sum_{i = 0}^{n^{(1)}} a^{(1)}_i \right) \st + \displaystyle\sum_{i = 0}^{n^{(1)}} \left(1 + \displaystyle\sum_{j = i}^{n^{(1)}} a^{(1)}_j \right) \down^{i+1} - \down \defeq{\cong} J^{(1)}
\end{equation}
by applying Theorem \ref{thm:int-star} with $J \defeq{\cong} 0$.
Also,
\begin{equation*}
\left [a^{(m)}_{n^{(m)}}, a^{(m)}_{n^{(m)}-1}, \ldots, a^{(m)}_0 \right]
\seq \left[a^{(m-1)}_{n^{(m-1)}}, a^{(m-1)}_{n^{(m-1)}-1}, \ldots, a^{(m-1)}_0 \right] 
\seq \cdots \seq \left[a^{(2)}_{n^{(2)}}, a^{(2)}_{n^{(2)}-1}, \ldots, a^{(2)}_0 \right]
\end{equation*}
is dicotic by Definition \ref{def:block} and Theorem \ref{lem:dicot}.
Hence, by applying Theorem \ref{cor:order-preserve}, the game value of (\ref{eq:cmw7ygqcgqi2}) is equal to the game value of
\begin{equation}
\label{eq:kfm8wrxljibn}
\left [a^{(m)}_{n^{(m)}}, a^{(m)}_{n^{(m)}-1}, \ldots, a^{(m)}_0 \right]
\seq \left[a^{(m-1)}_{n^{(m-1)}}, a^{(m-1)}_{n^{(m-1)}-1}, \ldots, a^{(m-1)}_0 \right] 
\seq \cdots \seq \left[a^{(2)}_{n^{(2)}}, a^{(2)}_{n^{(2)}-1}, \ldots, a^{(2)}_0 \right]
\seq J^{(1)}
\end{equation}
by replacing the rightmost component with $J^{(1)}$.

Also, we have $J^{(1)} < 0$ by Lemma \ref{lem:seq-neg}.
Therefore, we obtain
\begin{eqnarray*}
\lefteqn{\left[a^{(2)}_{n^{(2)}}, a^{(2)}_{n^{(2)}-1}, \ldots, a^{(2)}_0 \right] \seq J^{(1)}}\\
&=&  J^{(1)} + \left(1+\displaystyle\sum_{i = 0}^{n^{(2)}} a^{(2)}_i \right) \st + \displaystyle\sum_{i = 1}^{n^{(1)}+1} \left(1+ \displaystyle\sum_{j = 0}^{n^{(2)}} a^{(2)}_j \right)  \down^i + \displaystyle\sum_{i = 1}^{n^{(2)}} \left(1 + \displaystyle\sum_{j = i}^{n^{(2)}} a^{(2)}_j \right) \down^{n^{(1)}+1+i}\\
&\defeq{\cong}& J^{(2)}
\end{eqnarray*}
by applying Theorem \ref{thm:int-star} with $J \defeq{\cong} J^{(1)}$,
and 
\begin{equation*}
\left [a^{(m)}_{n^{(m)}}, a^{(m)}_{n^{(m)}-1}, \ldots, a^{(m)}_0 \right]
\seq \left[a^{(m-1)}_{n^{(m-1)}}, a^{(m-1)}_{n^{(m-1)}-1}, \ldots, a^{(m-1)}_0 \right] 
\seq \cdots \seq \left[a^{(3)}_{n^{(3)}}, a^{(3)}_{n^{(3)}-1}, \ldots, a^{(3)}_0 \right]
\end{equation*}
is dicotic by Definition \ref{def:block} and Theorem \ref{lem:dicot}.
Thus, by Theorem \ref{cor:order-preserve}, we can rewrite (\ref{eq:kfm8wrxljibn}) further as
\begin{equation*}
\left [a^{(m)}_{n^{(m)}}, a^{(m)}_{n^{(m)}-1}, \ldots, a^{(m)}_0 \right]
\seq \left[a^{(m-1)}_{n^{(m-1)}}, a^{(m-1)}_{n^{(m-1)}-1}, \ldots, a^{(m-1)}_0 \right] 
\seq \cdots \seq \left[a^{(3)}_{n^{(3)}}, a^{(3)}_{n^{(3)}-1}, \ldots, a^{(3)}_0 \right]
\seq J^{(2)},
\end{equation*}
and $J^{(2)} < 0$ also holds by Lemma \ref{lem:seq-neg}.

By repeating this process, we finally obtain the desired game value of (\ref{eq:cmw7ygqcgqi2}) in the form $c\st + \sum_{i = 1}^k b_i \down^i$
for some integer $k \geq 0$ and sequence $(c, b_1, b_2, \ldots, b_k)$ of non-negative integers.

To see the game value of (\ref{eq:mh9njmbt23ou}), we perform the first step as
\begin{equation*}
\left[a^{(1)}_{n^{(1)}}, a^{(1)}_{n^{(1)}-1}, \ldots, a^{(1)}_0 \right] \seq \st
= \left(\displaystyle\sum_{i = 1}^{n^{(1)}} a^{(1)}_i \right) \st + \displaystyle\sum_{i = 1}^{n^{(1)}} \left(1 + \displaystyle\sum_{j = i}^{n^{(1)}} a^{(1)}_j \right) \down^{i}  \defeq{\cong} J^{(1)}
\end{equation*}
by applying Theorem \ref{thm:int-star} with $J \defeq{\cong} \st$ instead of (\ref{eq:b17wbdx2owaa}),
and then join the process of the case of (\ref{eq:cmw7ygqcgqi2}).

Note that Theorem \ref{thm:int-star} gives a slightly stronger result than necessary to obtain the game value of (\ref{eq:slcsc43u2u2z}) and (\ref{eq:2vmcdrxoww2z}): in the case $J < 0$, the assertion $[a_n, a_{n-1}, \ldots, a_0] \seq J = \langle a_n, a_{n-1}, \ldots, a_0 \rangle_{J}$ also holds for $J$ that is not necessarily in the form of the sequential compounds of integers and stars as long as (\ref{eq:dnqwaieisfhc}) is satisfied.

As a result of Theorem \ref{thm:int-star}, we observe that
\begin{equation}
\label{eq:136lkbj4einf}
1 \seq \st = (-1) \seq \st \seq \st = \down,
\end{equation}
and if
\begin{equation*}
G_i \seq G_{i+1} \seq \cdots \seq G_n \seq J = c \st + \sum_{j = 1}^k b_j \down^j, \quad b_k \geq 1,
\end{equation*}
then
\begin{eqnarray}
\st \seq G_i \seq G_{i+1} \seq \cdots \seq G_n \seq J &=& (c+1) \st + \sum_{i = 1}^k (b_i+1) \down^i, \label{eq:poju4ae0ockf}\\
1 \seq G_i \seq G_{i+1} \seq \cdots \seq G_n \seq J &=& (c+1) \st + \sum_{i = 1}^k (b_i+1) \down^i, \label{eq:krlep0ujgnyb}\\
(-1) \seq G_i \seq G_{i+1} \seq \cdots \seq G_n \seq J &=& c \st + \sum_{i = 1}^k b_i \down^i + \down^{k+1}. \label{eq:8fv4lfimxw7e}
\end{eqnarray}
This yields a simpler algorithm to calculate the game value of (\ref{eq:slcsc43u2u2z}) (resp.~(\ref{eq:2vmcdrxoww2z})) as follows:
we start with $G \defeq{\cong} 1 \seq \st$ (resp.~$G \defeq{\cong} (-1) \seq \st \seq \st$), which is equal to $-0.1$ in the uptimal notation by (\ref{eq:136lkbj4einf});
then we repeatedly add $1, -1, \st$ to the left of $G$ appropriately to make the desired sequential compound.
By (\ref{eq:poju4ae0ockf})--(\ref{eq:8fv4lfimxw7e}), appending $1$ or $\st$ increases every digit of the uptimal notation by $1$ and flips the presence/absence of a star; appending $-1$ appends a digit $1$ to the end of the uptimal notation. 

\begin{example}
The game value of $(-1) \seq 1 \seq \st \seq (-1) \seq (-1) \seq (-1) \seq 1 \seq (-1) \seq \st \seq \st$
is obtained by starting with $G \defeq{\cong} (-1) \seq \st \seq \st$ and
then appending $1, -1, \st$ to the left of $G$ sequentially as follows:
\begin{eqnarray}
(-1) \seq \st \seq \st &=& -0.1, \nonumber\\
1 \seq (-1) \seq \st \seq \st &=& -0.2 + \st,\nonumber\\
(-1) \seq 1 \seq (-1) \seq \st \seq \st &=& -0.21 + \st,\nonumber\\
(-1) \seq (-1) \seq 1 \seq (-1) \seq \st \seq \st &=& -0.211 + \st,\nonumber\\
(-1) \seq (-1) \seq (-1) \seq 1 \seq (-1) \seq \st \seq \st &=& -0.2111 + \st,\nonumber\\
\st \seq (-1) \seq (-1) \seq (-1) \seq 1 \seq (-1) \seq \st \seq \st &=& -0.3222,\nonumber\\
1 \seq \st \seq (-1) \seq (-1) \seq (-1) \seq 1 \seq (-1) \seq \st \seq \st &=& -0.4333 + \st,\nonumber\\
(-1) \seq 1 \seq \st \seq (-1) \seq (-1) \seq (-1) \seq 1 \seq (-1) \seq \st \seq \st &=& -0.43331 + \st. \label{eq:k3lt7bukhbuy}
\end{eqnarray}

Similarly,
the game value of $(-1) \seq 1 \seq (-1) \seq \st \seq (-1) \seq 1 \seq \st$
is obtained as follows:
\begin{eqnarray}
1 \seq \st &=& -0.1,\nonumber\\
(-1) \seq 1 \seq \st  &=& -0.11,\nonumber\\
\st \seq (-1) \seq 1 \seq \st  &=& -0.22 + \st,\nonumber\\
(-1) \seq \st \seq (-1) \seq 1 \seq \st  &=& -0.221 + \st,\nonumber\\
1 \seq (-1) \seq \st \seq (-1) \seq 1 \seq \st  &=& -0.332,\nonumber\\
(-1) \seq 1 \seq (-1) \seq \st \seq (-1) \seq 1 \seq \st  &=& -0.3321. \label{eq:ngrwgijvthef}
\end{eqnarray}
\end{example}

\begin{example}
\label{ex:6seq}
We calculate the game values of the following six sequential compounds $G_1, G_2, \ldots, G_6$ and 
find the outcome class of their disjunctive sum $G_1 + G_2 + \cdots + G_6$ by adding their game values:
\begin{eqnarray*}
G_1 &\defeq{\cong}& (-1) \seq (-3) \seq (-1),\\
G_2 &\defeq{\cong}& (-2) \seq 1 \seq (-1) \seq 4,\\
G_3 &\defeq{\cong}& \st \seq \st \seq \st,\\
G_4 &\defeq{\cong}& 1 \seq (-1) \seq \st \seq 3 \seq (-1) \seq 1 \seq \st \seq \st \seq \st \seq \st \seq (-2) \seq 1,\\
G_5 &\defeq{\cong}& (-1) \seq 1 \seq (-1) \seq \st \seq (-1) \seq 1 \seq \st \seq 4 \seq (-1),\\
G_6 &\defeq{\cong}& \st \seq \st \seq \st \seq \st \seq \st \seq \st \seq (-2) \seq 2.
\end{eqnarray*}

The game value of $G_1$ is obtained as
\begin{equation*}
G_1
\cong (-1) \seq (-3) \seq (-1)
\eqlab{A}{\cong} (-1) \seq (-1) \seq (-1) \seq (-1) \seq (-1)
\eqlab{B}{\cong} -5,
\end{equation*}
where
(A) follows from Theorem \ref{thm:int},
and (B) follows from Theorem \ref{thm:int}.

The game value of $G_2$ is obtained as
\begin{equation*}
G_2
\cong (-2) \seq 1 \seq (-1) \seq 4
\eqlab{A}{\cong} (-1) \seq (-1) \seq 1 \seq (-1) \seq 1 \seq 3
\eqlab{B}{=} 3 + \frac{1}{2^0} + \frac{(-1)}{2^1} + \frac{1}{2^2} + \frac{(-1)}{2^3} + \frac{(-1)}{2^4}
= \frac{57}{16},
\end{equation*}
where
(A) follows from Theorem \ref{thm:int},
and (B) follows from Theorem \ref{thm:int2}.

The game value of $G_3$ is obtained as
\begin{equation*}
G_3 
\cong \st \seq \st \seq \st
\eqlab{A}{=} \st,
\end{equation*}
where (A) follows from Theorem \ref{thm:star} (ii).

We consider the game value of $G_4$.
We have 
\begin{equation}
\label{eq:rcknn3hs76w8}
(-2) \seq 1
\eqlab{A}{\cong} (-1) \seq (-1) \seq 1
\eqlab{B}{=} \frac{1}{2^0} + \frac{(-1)}{2^1} + \frac{(-1)}{2^2}
= \frac{1}{4},
\end{equation}
where
(A) follows from Theorem \ref{thm:int},
and (B) follows from Theorem \ref{thm:int2}.
Hence, the game value of $G_4$ is given as
\begin{eqnarray*}
G_4
&\cong& 1 \seq (-1) \seq \st \seq 3 \seq (-1) \seq 1 \seq \st \seq \st \seq \st \seq \st \seq (-2) \seq 1\\
&\eqlab{A}{=}& 1\seq (-1) \seq \st \seq 3 \seq (-1) \seq 1 \seq \st \seq \st \seq \st \seq \st \seq \frac{1}{4}\\
&\eqlab{B}=& \left(1 \seq (-1) \seq \st \seq 3 \seq (-1) \seq 1 \seq \st \seq \st \seq \st \seq \st \right) + \frac{1}{4}\\
&\eqlab{C}=& \left(1 \seq (-1) \seq \st \seq 3 \seq (-1) \seq 1 \seq \st \seq \st\right) + \frac{1}{4}\\
&\cong& -\left((-1) \seq 1 \seq \st \seq (-3) \seq 1 \seq (-1) \seq \st \seq \st\right) + \frac{1}{4}\\
&\eqlab{D}\cong& -\left((-1) \seq 1 \seq \st \seq (-1) \seq (-1) \seq (-1) \seq 1 \seq (-1) \seq \st \seq \st\right) + \frac{1}{4}\\
&\eqlab{E}=& 0.43331 + \st + \frac{1}{4},
\end{eqnarray*}
where
(A) follows from Corollary \ref{cor:order-preserve} and (\ref{eq:rcknn3hs76w8}),
(B) follows from Corollary \ref{cor:num-trans},
(C) follows from Theorem \ref{thm:star} (ii),
(D) follows from Theorem \ref{thm:int},
and (E) follows from (\ref{eq:k3lt7bukhbuy}).

We consider the game value of $G_5$.
We have
\begin{eqnarray}
4 \seq (-1)
&\eqlab{A}\cong& 1 \seq 1 \seq 1 \seq 1 \seq (-1)\nonumber\\
&\cong& -\left( (-1) \seq (-1) \seq (-1) \seq (-1) \seq 1 \right)\nonumber\\
&\eqlab{B}=& -\left( \frac{1}{2^0} + \frac{(-1)}{2^1} + \frac{(-1)}{2^2} + \frac{(-1)}{2^3} + \frac{(-1)}{2^4}\right)\nonumber\\
&=& -\frac{1}{16}, \label{eq:hv1vefa9il55}
\end{eqnarray}
where
(A) follows from Theorem \ref{thm:int},
and (B) follows from Theorem \ref{thm:int2}.
Hence, the game value of $G_5$ is given as
\begin{eqnarray*}
G_5
&\cong& (-1) \seq 1 \seq (-1) \seq \st \seq (-1) \seq 1 \seq \st \seq 4 \seq (-1)\\
&\eqlab{A}=& (-1) \seq 1 \seq (-1) \seq \st \seq (-1) \seq 1 \seq \st \seq \left(-\frac{1}{16}\right)\\
&\eqlab{B}=& \left((-1) \seq 1 \seq (-1) \seq \st \seq (-1) \seq 1 \seq \st\right) - \frac{1}{16}\\
&\eqlab{C}=& -0.3321 - \frac{1}{16},
\end{eqnarray*}
where
(A) follows from Corollary \ref{cor:order-preserve} and (\ref{eq:hv1vefa9il55}),
(B) follows from Corollary \ref{cor:num-trans},
and (C) follows from (\ref{eq:ngrwgijvthef}).

We consider the game value of $G_6$.
We have
\begin{eqnarray}
\label{eq:f956i202ir8g}
(-2) \seq 2
\eqlab{A}\cong (-1) \seq (-1) \seq 1 \seq 1
\eqlab{B}= 1 + \frac{1}{2^0} + \frac{(-1)}{2^1} + \frac{(-1)}{2^2}
= \frac{5}{4},
\end{eqnarray}
where
(A) follows from Theorem \ref{thm:int},
and (B) follows from Theorem \ref{thm:int2}.
Hence, the game value of $G_6$ is given as
\begin{eqnarray*}
G_6
&\cong& \st \seq \st \seq \st \seq \st \seq \st \seq \st \seq (-2) \seq 2\\
&\eqlab{A}=& \st \seq \st \seq \st \seq \st \seq \st \seq \st \seq \frac{5}{4}\\
&\eqlab{B}=& \left(\st \seq \st \seq \st \seq \st \seq \st \seq \st\right) + \frac{5}{4}\\
&\eqlab{C}=& \left(\st \seq \st \seq \st \seq \st\right) + \frac{5}{4}\\
&\eqlab{D}=& \left(\st \seq \st\right) + \frac{5}{4}\\
&=& 0 + \frac{5}{4}\\
&=& \frac{5}{4},
\end{eqnarray*}
where
(A) follows from Corollary \ref{cor:order-preserve} and (\ref{eq:f956i202ir8g}),
(B) follows from Corollary \ref{cor:num-trans},
(C) follows from Theorem \ref{thm:star} (ii),
and (D) follows from Theorem \ref{thm:star} (ii).

Finally, by summing up the game values, we obtain
\begin{eqnarray*}
\sum_{i = 1}^6 G_i
= (-5) + \frac{57}{16} + \st + \left( 0.43331 + \st + \frac{1}{4}\right) + \left(- 0.3321 - \frac{1}{16}\right) + \frac{5}{4}
= 0.10121
\eqlab{A}{\parallel} 0,
\end{eqnarray*}
where (A) follows from Propositions \ref{prop:uptimal-comp} and \ref{prop:uptimal-star}.
Consequently, we obtain $o\left(\sum_{i = 1}^6 G_i \right) = \mathscr{N}$ by Corollary \ref{cor:order-outcome}.
\end{example}

\section{Conclusion}
\label{sec:conclusion}

This paper discussed sequential compounds of partizan games defined by \cite{SU93}.
We first showed three theorems on the general properties of sequential compounds,
and then by applying them, we gave integers $c \geq 0, k \geq 0$ and a sequence $(b_1, b_2, \ldots, b_k)$ of integers such that
\begin{equation}
\label{eq:bw0imc65otco}
G = c\st + \sum_{i = 1}^k b_i \up^k
\end{equation}
for a given sequential compound
\begin{equation}
\label{eq:pgsh2xy6ouif}
G \defeq{\cong} G_1 \seq G_2 \seq \cdots \seq G_n, \quad \,\,\text{where}\,\, G_1, G_2, \ldots, G_n \in \{1, -1, \st\}.
\end{equation}
For a given disjunctive sum $H \defeq{\cong} H_1 + H_2 + \cdots + H_n$ of sequential compounds (i.e., each $H_i$ is in the form (\ref{eq:pgsh2xy6ouif})),
we can determine the outcome class of $H$ by rewriting every $H_i$ in the form (\ref{eq:bw0imc65otco}) and
summing them up as shown in Example \ref{ex:6seq}.

The future work includes the following topics.
\begin{itemize}
\item The generalization to the case of allowing $\st k$, defined recursively as follows, as components of sequential compounds:
\begin{equation*}
\st k \defeq{\cong} \left\{0, \st, \st2, \ldots, \st(k-1) \relmiddle| 0, \st, \st2, \ldots, \st(k-1) \right\}
\end{equation*}
for an integer $k \geq 1$.

\item The generalization to sequential compounds on a directed acyclic graph:
every node $v$ is assigned a game $G_v$; a player can play only on $G_x$ such that for all other nodes $y$ reachable to $x$, the player has no option in $G_y$.
Note that $G_1 \seq G_2 \seq \cdots \seq G_n$ in this paper is a particular case in which the graph is a path graph.

\item The discussion in the case of adopting (\ref{eq:9znzjudvuld5}) as the definition of sequential compounds instead of Definition \ref{def:seq}.
\end{itemize}

 \appendix

\section{Proof of Theorem \ref{thm:int-star}}
\label{sec:proof-int-star}

We fix an integer $n \geq 0$, a sequence $(a_0, a_1, \ldots, a_n)$ of non-negative integers, and $J \in \short$, and assume that one of the conditions (a)--(c) of Theorem \ref{thm:int-star} holds.
We prove that the assertion of the theorem holds for the given $(n, a_0, a_1, \ldots, a_n, J)$ in this section.

First, we state some observations and lemmas for the proof as follows.
We confirm
\begin{align}
&[0] \seq 0 \cong \st \cong \langle 0 \rangle_0, \label{eq:rgour9uavjr2} \\
&[0] \seq \st \cong \st \seq \st = 0 \cong \langle 0 \rangle_{\st}, \label{eq:wsptsdy7tlwx} \\
&\deg( \langle a_n, a_{n-1}, \ldots, a_0 \rangle_J ) \geq 1 \label{lem:seq-deg}
\end{align}
directly from (\ref{eq:xhzxoe0bbtgq}).

The proof of Theorem \ref{thm:int-star} also relies on the following Lemmas \ref{lem:ub9enb4w0s6a}--\ref{lem:e1xr1ufi5spj}.
The proofs of them are deferred to Appendices \ref{subsec:proof-ub9enb4w0s6a}--\ref{subsec:proof-e1xr1ufi5spj}, respectively.

\begin{lemma}
\label{lem:ub9enb4w0s6a}
The following statements (i) and (ii) hold.
\begin{enumerate}[(i)]
\item If $n \geq 1$, then $\langle a_{n-1}, \ldots, a_0\rangle_{J} - \langle a_n, a_{n-1}, \ldots, a_0 \rangle_{J} \not\leq 0$.
\item If $n = 0$, then $J - \langle a_0 \rangle_{J} \not\leq 0$.
\end{enumerate}
\end{lemma}

 \begin{lemma}
 \label{lem:3rp1dfrhr0nc}
For any Left option $\langle a_n, a_{n-1}, \ldots, a_0 \rangle_{J}^L$ of $\langle a_n, a_{n-1} \ldots, a_0 \rangle_{J}$, the following statements (i) and (ii) hold.
\begin{enumerate}[(i)]
\item If $a_n = a_{n-1} = \cdots = a_{p+1} = 0, a_p \geq 1$ for an integer $0 \leq p \leq n$, then $\langle a_n, a_{n-1}, \ldots, a_0 \rangle_{J}^L \leq \langle a_p-1, a_{p-1}, \ldots a_0 \rangle_{J}$.
\item If $a_n = a_{n-1} = \cdots = a_0 = 0$, then $\langle a_n, a_{n-1}, \ldots, a_0 \rangle_{J}^L \leq J$.
\end{enumerate}
\end{lemma}

\begin{lemma}
\label{lem:q7g997f5n2b0}
The following statements (i) and (ii) hold.
\begin{enumerate}[(i)]
\item If $a_n = a_{n-1} = \cdots = a_{p+1} = 0, a_p \geq 1$ for an integer $0 \leq p \leq n$, then 
$\langle a_p-1, a_{p-1}, a_{p-2}, \ldots, a_0 \rangle_{J} - \langle \underbrace{0, \ldots, 0}_{n-p}, a_p, a_{p-1}, \ldots, a_0 \rangle_{J} \parallel 0$.
\item $J - \langle \underbrace{0, \ldots, 0}_{n+1} \rangle_{J} \parallel 0$.
\end{enumerate}
\end{lemma}

\begin{lemma}
\label{lem:e1xr1ufi5spj}
If $n \geq 1$, then 
$\langle a_{n-1}, \ldots, a_0\rangle_{J}
= \langle 0, a_{n-1}, \ldots, a_0\rangle_{J} - \down^{d}$, where $d \defeq= \deg(\langle a_n, a_{n-1}, \ldots, a_0\rangle_{J} )$.
\end{lemma}

Now, we move on to the proof of Theorem \ref{thm:int-star}.
We prove
\begin{equation*}
G \defeq{\cong} \left( [a_n, a_{n-1}, \ldots, a_0] \seq J \right) - \langle a_n, a_{n-1}, \ldots, a_0\rangle_{J} = 0
\end{equation*}
by induction on $n + \sum_{i = 0}^n a_i$.
More specifically, we prove $G \geq 0$ and $G \leq 0$ in Subsections \ref{subsec:proof-pos} and \ref{subsec:proof-neg} assuming
that the assertion holds for any $(n', a'_0, a'_1, \ldots, a'_{n'}, J')$ such that $n' + \sum_{i = 0}^{n'} a'_i < n + \sum_{i = 0}^n a_i$.

\subsection{Proof of $G \geq 0$}
\label{subsec:proof-pos}

To prove $G \geq 0$, we show that Left wins $G$ playing second, that is, Left wins any Right option of $G$ playing first.
In other words, we show that for any $G^R \in G^{\mathcal{R}}$, we have $G^R \not \leq 0$.
We consider the following two cases separately: the case where Right plays on $ [a_n, a_{n-1}, \ldots, a_0] \seq J$ next and
the case where Right plays on $-\langle a_n, a_{n-1}, \ldots, a_0\rangle_{J}$ next.

\subsubsection{Right plays on $[a_n, a_{n-1}, \ldots, a_0] \seq J $}

We show that the Right option $G^R$ corresponding to the Right's move on $[a_n, a_{n-1}, \ldots, a_0] \seq J$ satisfies $G^R \not\leq 0$ dividing into the following two cases:
the case $n \geq 1$ and the case $n = 0$.

\begin{itemize}
\item The case $n \geq 1$: We have
\begin{align*}
G^R \cong \left( [a_{n-1}, \ldots, a_0] \seq J \right) - \langle a_n, a_{n-1}, \ldots, a_0 \rangle_{J}
\eqlab{A}{=} \langle a_{n-1}, \ldots, a_0\rangle_{J} - \langle a_n, a_{n-1}, \ldots, a_0 \rangle_{J}
\eqlab{B}{\not\leq} 0,
\end{align*}
where (A) follows from (\ref{eq:wsptsdy7tlwx}) and $a_0 = 0$ of Theorem \ref{thm:int-star} (b) for the case $(n, J) = (1, \st)$,\footnote{In the case $(n, J) = (1, \st)$, we cannot apply the induction hypothesis to $[a_{n-1}, \ldots, a_0] \seq J$ because none of the conditions (a)--(c) of Theorem \ref{thm:int-star} is satisfied.}
and from the induction hypothesis for the other case,
and (B) follows from Lemma \ref{lem:ub9enb4w0s6a} (i).

\item The case $n = 0$: Then since $[a_n, a_{n-1}, \ldots, a_0] \seq J  \cong (a_0 \seq \st) \seq J$, we have
\begin{align*}
G^R \cong J - \langle a_0\rangle_{J} \eqlab{A}{\not\leq} 0,
\end{align*}
where
(A) follows from Lemma \ref{lem:ub9enb4w0s6a} (ii).
\end{itemize}

\subsubsection{Right plays on $-\langle a_n, a_{n-1}, \ldots, a_0 \rangle_{J}$}

A Right's move on $-\langle a_n, a_{n-1}, \ldots, a_0 \rangle_{J}$ results in
\begin{equation*}
G^R \cong \left( [a_n, a_{n-1}, \ldots, a_0] \seq J \right) - \langle a_n, a_{n-1}, \ldots, a_0 \rangle_{J}^L
\end{equation*}
for some Left option $\langle a_n, a_{n-1}, \ldots, a_0 \rangle_{J}^L$ of $\langle a_n, a_{n-1}, \ldots, a_0 \rangle_{J}$.
We show that $G^R$ has a Left option $G^{RL}$ such that $G^{RL} \geq 0$ for the following two cases separately:
the case where $a_n = a_{n-1} = \cdots = a_{p+1} = 0, a_p \geq 1$ for some $0 \leq p \leq n$ and
the case $a_n = a_{n-1} = \cdots = a_0 = 0$.

\begin{itemize}
\item The case where $a_n = a_{n-1} = \cdots = a_{p+1} = 0, a_p \geq 1$ for some $0 \leq p \leq n$:
The game $G^R$ has a Left option
\begin{align*}
G^{RL} 
&\defeq{\cong} \left( [a_p-1, a_{p-1}, a_{p-2}, \ldots, a_0] \seq J \right) - \langle a_n, a_{n-1}, \ldots, a_0 \rangle_{J}^L\\
&\eqlab{A}{\geq} \left( [a_p-1, a_{p-1}, a_{p-2}, \ldots, a_0] \seq J \right) - \langle a_p-1, a_{p-1}, a_{p-2}, \ldots, a_0 \rangle_{J}\\
&\eqlab{B}{=} \langle a_p-1, a_{p-1}, a_{p-2}, \ldots, a_0\rangle_{J} - \langle a_p-1, a_{p-1}, a_{p-2}, \ldots, a_0 \rangle_{J}\\
&= 0,
\end{align*}
where 
(A) follows from Lemma \ref{lem:3rp1dfrhr0nc} (i),
and (B) follows from (\ref{eq:rgour9uavjr2}) for the case $(p, a_p, J) = (0, 1, 0)$\footnote{In the case $(p, a_p, J) = (0, 1, 0)$, we cannot apply the induction hypothesis to $[a_p-1, a_{p-1}, a_{p-2}, \ldots, a_0] \seq J$ because none of the conditions (a)--(c) of Theorem \ref{thm:int-star} is satisfied.}
and the induction hypothesis for the other case.

\item The case $a_n = a_{n-1} = \cdots = a_0 = 0$:
Then since $[a_n, a_{n-1}, \ldots, a_0] \seq J  \cong (-1) \seq (-1) \seq \cdots \seq (-1) \seq \st \seq J$,
the game $G^R$ has a Left option
\begin{align*}
G^{RL} \defeq{\cong} J - \langle a_n, a_{n-1}, \ldots, a_0 \rangle_{J}^L
\eqlab{A}\geq J - J = 0,
\end{align*}
where
(A) follows from Lemma \ref{lem:3rp1dfrhr0nc} (ii).
\end{itemize}

\subsection{Proof of $G \leq 0$}
\label{subsec:proof-neg}

Let $d \defeq{=} \deg(\langle a_n, a_{n-1}, \ldots, a_0 \rangle_{J})$, and define
$(q, p_1, p_2, \ldots, p_d)$ so that
\begin{equation}
\label{eq:d34flbi8m2te}
\langle a_n, a_{n-1}, \ldots, a_0 \rangle_{J} \cong q \st + \sum_{i = 1}^{d} p_i \down^i.
\end{equation}
To prove $G \leq 0$, it suffices to prove the following Proposition \ref{prop:proof-neg}.
\begin{proposition}
\label{prop:proof-neg}
Let $(q', p'_1, \ldots, p'_d)$ be an arbitrary sequence of non-negative integers such that
\begin{eqnarray}
&&H' \defeq{\cong} q' \st + \sum_{i = 1}^{d} p'_i \down^i \geq \langle a_n, a_{n-1}, \ldots, a_0 \rangle_{J}, \label{eq:4zvhmrup5k1l}\\
&&p_i \geq p'_i \quad\text{for any}\,\,i = 1, 2, \ldots, d, \label{eq:lpwhcjc3w1re}\\
&&p'_1 \geq p'_2 \geq \cdots \geq p'_{d}. \label{eq:sl195b7g7f7t}
\end{eqnarray}
Then we have
$G' \defeq{\cong} \left( [a_n, a_{n-1}, \ldots, a_0] \seq J \right) - H' \leq 0$.
\end{proposition}

Applying Proposition \ref{prop:proof-neg} with $(q', p'_1, \ldots, p'_d) \defeq{=} (q, p_1, p_2, \ldots, p_d)$, we obtain the desired result:
\begin{eqnarray*}
G &\cong& \left( [a_n, a_{n-1}, \ldots, a_0] \seq J \right) - \langle a_n, a_{n-1}, \ldots, a_0\rangle_{J} \\
&\cong& \left( [a_n, a_{n-1}, \ldots, a_0] \seq J \right) - \left(q \st + \sum_{i = 1}^{d} p_i \down^i\right) \\
&\cong& \left( [a_n, a_{n-1}, \ldots, a_0] \seq J \right) - \left(q' \st + \sum_{i = 1}^{d} p'_i \down^i\right) \\
&\cong& \left( [a_n, a_{n-1}, \ldots, a_0] \seq J \right) - H' \\
&\eqlab{A}\leq& 0,
\end{eqnarray*}
where
(A) follows from Proposition \ref{prop:proof-neg} because
it is easily seen that $(q', p'_1, \ldots, p'_d) = (q, p_1, p_2, \ldots, p_d)$ satisfies (\ref{eq:4zvhmrup5k1l}) and (\ref{eq:lpwhcjc3w1re}),
and the assumption (\ref{eq:sl195b7g7f7t}) is seen from (\ref{eq:3dj9mrwnisxe}) and (\ref{eq:xhzxoe0bbtgq}).

We prove Proposition \ref{prop:proof-neg} by induction on $\sum_{i = 1}^d p'_i$.
We fix a given $(q', p'_1, p'_2, \ldots, p'_d)$ satisfying (\ref{eq:4zvhmrup5k1l})--(\ref{eq:sl195b7g7f7t}).
Because $G' \leq 0$ is shown by Lemma \ref{lem:seq-neg} in the case $H' \cong 0$, we assume $H' \not\cong 0$ and 
show that Right wins any Left option $G'^L$ of $G'$ playing first for the following two cases separately:
the case where Left plays on $[a_n, a_{n-1}, \ldots, a_0] \seq J$ next and the case where Left plays on $-H'$ next.

\subsubsection{Left plays on $[a_n, a_{n-1}, \ldots, a_0] \seq J $}

We show that the Left option $G'^L$ corresponding to the Left's move on $[a_n, a_{n-1}, \ldots, a_0] \seq J$ satisfies $G'^L \not\geq 0$
dividing into the following two cases: the case where $a_n = a_{n-1} = \cdots = a_{p+1} = 0, a_p \geq 1$ for some $0 \leq p \leq n$ and the case $a_n = a_{n-1} = \cdots = a_0 = 0$.

\begin{itemize}
\item The case where $a_n = a_{n-1} = \cdots = a_{p+1} = 0, a_p \geq 1$ for some $0 \leq p \leq n$:
We have
\begin{eqnarray*}
G'^L
&\cong& \left( [a_p-1, a_{p-1}, a_{p-2}, \ldots, a_0] \seq J \right) - H'\\
&\eqlab{A}\leq& \left( [a_p-1, a_{p-1}, a_{p-2}, \ldots, a_0] \seq J \right) - \langle a_n, a_{n-1}, \ldots, a_0 \rangle_{J}\\
&\eqlab{B}=& \langle a_p-1, a_{p-1}, a_{p-2}, \ldots, a_0 \rangle_{J} - \langle a_n, a_{n-1}, \ldots, a_0 \rangle_{J}\\
&\eqlab{C}=& \langle a_p-1, a_{p-1}, a_{p-2}, \ldots, a_0 \rangle_{J} - \langle \underbrace{0, \ldots, 0}_{n-p}, a_p, a_{p-1}, \ldots, a_0 \rangle_{J}\\[-2em]
&\eqlab{D}{\not\geq}& 0,
\end{eqnarray*}
where
(A) follows from (\ref{eq:4zvhmrup5k1l}),
(B) follows  from (\ref{eq:rgour9uavjr2}) for the case $(p, a_p, J) = (0, 1, 0)$ and the induction hypothesis for the other case,
(C) follows from $a_n = a_{n-1} = \cdots = a_{p+1} = 0$,
and (D) follows from Lemma \ref{lem:q7g997f5n2b0} (i).
This shows $G'^L \not\geq 0$ as desired.

\item The case $a_n = a_{n-1} = \cdots = a_0 = 0$:
Then since $[a_n, a_{n-1}, \ldots, a_0] \seq J  \cong (-1) \seq (-1) \seq \cdots \seq (-1) \seq \st \seq J$, we have
\begin{eqnarray*}
G'^L
\cong J - H'
\eqlab{A}\leq J - \langle a_n, a_{n-1}, \ldots, a_0 \rangle_{J}
\cong J - \langle \underbrace{0, \ldots, 0}_{n+1} \rangle_{J}
\eqlab{B}{\not\geq} 0,
\end{eqnarray*}
where
(A) follows from (\ref{eq:4zvhmrup5k1l}),
and (B) follows from Lemma \ref{lem:q7g997f5n2b0} (ii).
This shows $G'^L \not\geq 0$ as desired.
\end{itemize}

\subsubsection{Left plays on $-H'$}

A Left's move on $-H'$ of $G'$ results in
\begin{equation}
\label{eq:uqbmtxjcsxfd}
G'^L \cong \left( [a_n, a_{n-1}, \ldots, a_0] \seq J \right) - H'^R
\end{equation}
for some $H'^R \in H'^{\mathcal{R}}$.

If $H' \cong \st$, then $H'^R \cong 0$ and thus
\begin{equation*}
G'^L \eqlab{A}\cong \left( [a_n, a_{n-1}, \ldots, a_0] \seq J \right) - H'^R\cong \left( [a_n, a_{n-1}, \ldots, a_0] \seq J \right) \eqlab{B}< 0,
\end{equation*}
where (A) follows from (\ref{eq:uqbmtxjcsxfd}),
and (B) follows from Lemma \ref{lem:seq-neg}.
We consider the case $H' \not\cong \st$, in which $\deg(H') \geq 1$.
Then since $H'^R \not\cong 0$ and $H'^R \cong q'' \st + \sum_{i = 1}^d p''_i \down^i$ for some sequence $(q'', p''_1, p''_2, \ldots, p''_d)$ of non-negative integers,
we have $H'^R \parallel 0$ or $H'^R < 0$ by Propositions \ref{prop:uptimal-comp} and \ref{prop:uptimal-star}.
If $H'^R \parallel 0$, then $G'^L \not\geq 0$ holds as desired because
\begin{align*}
G'^L \eqlab{A}\cong \left( [a_n, a_{n-1}, \ldots, a_0] \seq J \right) - H'^R
\eqlab{B}< - H'^R
\parallel 0,
\end{align*}
where 
(A) follows from (\ref{eq:uqbmtxjcsxfd}),
and (B) follows from Lemma \ref{lem:seq-neg}.
Thus, we may also assume $H'^R < 0$ so that $\deg(H'^R) \geq 1$.

Since $\deg(H') \geq 1$, we have $H'^R = H' + \st$ or $H'^R \geq H' - \down^{d'}$ by Lemma \ref{lem:uptimal-option} (ii), where $d' \defeq{=} \deg(H')$.
We consider the following two cases separately: the case $H'^R = H' + \st$ and the case $H'^R \geq H' - \down^{d'}$.

\begin{itemize}
\item The case $H'^R = H' + \st$: It follows $G'^L \not\geq 0$ as desired because
\begin{equation*}
G'' \defeq\cong \left( [a_n, a_{n-1}, \ldots, a_0] \seq J \right) - \left(H' + \st\right)
\end{equation*}
satisfies
\begin{equation*}
G'' = \left( [a_n, a_{n-1}, \ldots, a_0] \seq J \right) - H'^R \cong G'^L,
\end{equation*}
and has a Right option $G''^{R}$ such that
\begin{eqnarray}
G''^{R}
&\eqlab{A}=& \left( [a_n, a_{n-1}, \ldots, a_0] \seq J \right) - \left( \left(H' - \sum_{i = 1}^{d'}\down^i + \st \right) + \st \right) \nonumber\\
&=& \left( [a_n, a_{n-1}, \ldots, a_0] \seq J \right) - \left( H' - \sum_{i = 1}^{d'}\down^i \right) \nonumber\\
&\cong& \left( [a_n, a_{n-1}, \ldots, a_0] \seq J \right) - \left( q' \st + \sum_{i = 1}^{d'} p'_i \down^i - \sum_{i = 1}^{d'}\down^i \right) \nonumber\\
&=& \left( [a_n, a_{n-1}, \ldots, a_0] \seq J \right) - \left( q' \st + \sum_{i = 1}^{d'} (p'_i-1) \down^i\right) \nonumber\\
&\eqlab{B}\cong& \left( [a_n, a_{n-1}, \ldots, a_0] \seq J \right) - \left(q'' \st + \sum_{i = 1}^{d} p''_i \down^i \right) \label{eq:josfyvqn40ik}\\
&\eqlab{C}\leq& 0, \nonumber
\end{eqnarray}

where
(A) follows from (\ref{eq:9f1ueuce84aw}),
(B) follows by defining the sequence $(q'', p''_1, p''_2, \ldots, p''_d)$ of non-negative integers so that the identity holds,
and (C) follows by applying the induction hypothesis because $(q'', p''_1, p''_2, \ldots, p''_d)$ satisfies (\ref{eq:4zvhmrup5k1l})--(\ref{eq:sl195b7g7f7t}), the proof of which is
deferred to Appendix \ref{subsec:proof-sirsnip1kywb} as the proof of the following lemma.

\begin{lemma}
\label{lem:sirsnip1kywb}
The sequence $(q'', p''_1, p''_2, \ldots, p''_d)$ defined by (\ref{eq:josfyvqn40ik}) satisfies (\ref{eq:4zvhmrup5k1l})--(\ref{eq:sl195b7g7f7t}).
\end{lemma}

\item The case $H'^R \geq H' - \down^{d'}$: 
It suffices to show that
\begin{equation}
\label{eq:ae0kzz634oy7}
G'' \defeq{\cong} \left( [0, a_{n-1}, \ldots, a_0] \seq J \right) - \left( H' - \down^{d'} \right) \not\geq 0
\end{equation}
by showing that $G''$ has a Right option $G''^R$ such that $G''^R \leq 0$
because the desired result $G'^L \not\geq 0$ is implied as
\begin{equation*}
G'^L \eqlab{A}\cong \left( [0, a_{n-1}, \ldots, a_0] \seq J \right) - H'^R
\leq \left( [0, a_{n-1}, \ldots, a_0] \seq J \right) - \left( H' - \down^{d'} \right)
\eqlab{B}\cong G'' \not\geq 0,
\end{equation*}
where
(A) follows from (\ref{eq:uqbmtxjcsxfd}),
and (B) follows from (\ref{eq:ae0kzz634oy7}).
We show (\ref{eq:ae0kzz634oy7}) dividing into the following two cases: the case $p_d = 1$ and the case $p_d \geq 2$.

\begin{itemize}
\item The case $p_d = 1$: This case is possible only if
\begin{equation}
\label{eq:9wf8kgfgeiy3}
a_n = 0
\end{equation}
by (\ref{eq:xhzxoe0bbtgq}).
We see $n \geq 1$ for the following three cases separately: the case $J \cong 0, a_0 \geq 1$, the case $J \cong \st, n \geq 1, a_0 = 0$, and the case $J < 0$.
\begin{itemize}
\item The case $J \cong 0, a_0 \geq 1$: By (\ref{eq:9wf8kgfgeiy3}) and $a_0 \geq 1$.
\item The case $J \cong \st, n \geq 1, a_0 = 0$: Directly from the assumption $n \geq 1$.
\item The case $J < 0$: If we assume $n = 0$, then
\begin{equation*}
q \st + \sum_{i = 1}^{d} p_i \down^i
\cong \langle a_n, a_{n-1}, \ldots, a_0 \rangle_J 
\cong \langle a_0 \rangle_J
\eqlab{A}\cong \langle 0 \rangle_J
\eqlab{B}\cong J + \st + \sum_{i = 1}^k \down^i
\eqlab{C}\cong (c+1)\st + \sum_{i = 1}^k (b_i+1) \down^i,
\end{equation*}
where 
(A) follows from (\ref{eq:9wf8kgfgeiy3}),
(B) follows from the third case of (\ref{eq:xhzxoe0bbtgq}),
and (C) follows from (\ref{eq:dnqwaieisfhc}).
Comparing both side, we obtain
$d = \deg(\langle a_n, a_{n-1}, \ldots, a_0 \rangle_J )= k$ and $p_d = b_k + 1 \geq 1+1 \geq 2$,
which conflict with $p_d = 1$.
\end{itemize}

Since $n \geq 1$ as shown above, $G''$ has a Right option
\begin{eqnarray*}
\left( [a_{n-1}, \ldots, a_0] \seq J \right) - \left( H' - \down^{d'} \right)
&\eqlab{A}{=}& \langle a_{n-1}, \ldots, a_0\rangle_{J} - H' + \down^{d'}\\
&\eqlab{B}=& \left(\langle 0, a_{n-1}, \ldots, a_0\rangle_{J} - \down^{d} \right) - H' + \down^{d'}\\
&\eqlab{C}=& \langle a_n, a_{n-1}, \ldots, a_0\rangle_{J} - H' - \down^{d} + \down^{d'}\\
&\eqlab{D}{\leq}& H' - H' + \down^{d} - \down^{d'}\\
&=& \down^{d} - \down^{d'}\\
&\eqlab{E}{\leq}& 0,
\end{eqnarray*}
where
(A) follows from (\ref{eq:wsptsdy7tlwx}) for the case $(n, J) = (1, \st)$ and the induction hypothesis for the other case,
(B) follows from Lemma \ref{lem:e1xr1ufi5spj},
(C) follows from (\ref{eq:9wf8kgfgeiy3}),
(D) follows from (\ref{eq:4zvhmrup5k1l}),
and (E) follows from Proposition \ref{prop:uptimal-comp} and $d' \leq d$.
This shows $G'' \not\geq 0$ as desired.

\item The case $p_d \geq 2$:
By $\deg(H') \geq 1$, the game $G''$
has a Right option
\begin{eqnarray}
G''^{R} &=& \left( [a_n, a_{n-1}, \ldots, a_0] \seq J \right) - \left(H' - \down^{d'} - \sum_{i = 1}^{d''} \down^i + \st\right) \nonumber\\
&\eqlab{A}\cong& \left( [a_n, a_{n-1}, \ldots, a_0] \seq J \right) - \left(q'' \st + \sum_{i = 1}^{d} p''_i \down^i \right) \label{eq:bk22tdopkqc2}\\
&\eqlab{B}\leq& 0, \nonumber
\end{eqnarray}
where 
$d'' \defeq{=} d'$ if $p_{d'} \geq 2$, and $d'' \defeq{=} d'-1$ if $p_{d'} = 1$,
(A) follows by defining the sequence $(q'', p''_1, p''_2, \ldots, p''_d)$ of non-negative integers so that the identity holds,
and (B) follows by applying the induction hypothesis because $(q'', p''_1, p''_2, \ldots, p''_d)$ satisfies (\ref{eq:4zvhmrup5k1l})--(\ref{eq:sl195b7g7f7t}),
the proof of which is deferred to Appendix \ref{subsec:proof-qdnup6wlyl8k} as the proof of the following lemma.

\begin{lemma}
\label{lem:qdnup6wlyl8k}
The sequence $(q'', p''_1, p''_2, \ldots, p''_d)$ defined by (\ref{eq:bk22tdopkqc2}) satisfies (\ref{eq:4zvhmrup5k1l})--(\ref{eq:sl195b7g7f7t}).
\end{lemma}
\end{itemize}
\end{itemize}

\section{Proofs of Lemmas}

\subsection{Proof of Lemma \ref{lem:uptimal-option}}
\label{subsec:proof-uptimal-option}

\begin{proof}[Proof of Lemma \ref{lem:uptimal-option}]
(Proof of (i))
Choose $G^L \in G^{\mathcal{L}}$ arbitrarily.
The Left option $G^L$ is reached from $G$ by a move on either $\st$ or $\down^{d'}$ for some $d' \leq d$.
In the former case, we have
\begin{equation*}
G^L
\cong (c-1) \cdot \st + \sum_{i = 1}^k b_i \cdot \down^i
= G + \st
\eqlab{A}\leq G - \sum_{i = 1}^d \down^i + \st,
\end{equation*}
where
(A) follows from Proposition \ref{prop:uptimal-comp} (i).
In the latter case, we have
\begin{equation*}
G^L
\eqlab{A}= G - \sum_{i = 1}^{d'} \down^i + \st
\eqlab{B}\leq G - \sum_{i = 1}^{d} \down^i + \st,
\end{equation*}
where (A) follows from (\ref{eq:9f1ueuce84aw}),
and (B) follows from $d' \leq d$ and Proposition \ref{prop:uptimal-comp} (i).

(Proof of (ii))
Choose $G^R \in G^{\mathcal{R}}$ arbitrarily.
The Right option $G^R$ is reached from $G$ by a move on either $\st$ or $\down^{d'}$ for some $d' \leq d$.
In the former case, we have $G^R = G + \st$.
In the latter case, we have
\begin{equation*}
G^R
\eqlab{A}= G - \down^{d'}
\eqlab{B}\geq  G - \down^{d},
\end{equation*}
where (A) follows from (\ref{eq:9f1ueuce84aw}),
and (B) follows from $d' \leq d$ and Proposition \ref{prop:uptimal-comp} (ii).
\end{proof}

\subsection{Proof of Lemma \ref{lem:seq}}
\label{subsec:proof-seq}

\begin{proof}[Proof of Lemma \ref{lem:seq}]
(Proof of (i))
We first show $G \seq 0 \cong G$ by induction on $\birth(G)$.
For the base case $G \cong 0$, we have
\begin{equation*}
G \seq 0 \cong 0 \seq 0 \eqlab{A}\cong 0 \cong G,
\end{equation*}
where (A) follows from (\ref{eq:r3xyoug24857}).
For the induction step for $G \not\cong 0$, we have
\begin{eqnarray*}
(G \seq 0)^{\mathcal{L}}
&\eqlab{A}{=}& \begin{cases}
G^{\mathcal{L}} \seq 0 &\,\,\text{if}\,\, G^{\mathcal{L}} \neq \emptyset,\\
0^{\mathcal{L}} &\,\,\text{if}\,\, G^{\mathcal{L}} = \emptyset
\end{cases} \\
&\eqlab{B}{=}& \begin{cases}
G^{\mathcal{L}} &\,\,\text{if}\,\, G^{\mathcal{L}} \neq \emptyset,\\
0^{\mathcal{L}} &\,\,\text{if}\,\, G^{\mathcal{L}} = \emptyset
\end{cases} \\
&=& \begin{cases}
G^{\mathcal{L}} &\,\,\text{if}\,\, G^{\mathcal{L}} \neq \emptyset,\\
\emptyset &\,\,\text{if}\,\, G^{\mathcal{L}} = \emptyset
\end{cases} \\
&=& G^{\mathcal{L}},
\end{eqnarray*}
where
(A) follows from (\ref{eq:r3xyoug24857}),
and (B) follows from the induction hypothesis.
Similarly, we obtain
$(G \seq 0)^{\mathcal{R}} = G^{\mathcal{R}}$.
This concludes $G \seq 0 \cong G$.

We next show $0 \seq G \cong G$.
We have
\begin{eqnarray*}
(0 \seq G)^{\mathcal{L}}
\eqlab{A}{=}
\begin{cases}
0^{\mathcal{L}} \seq G &\,\,\text{if}\,\, 0^{\mathcal{L}} \neq \emptyset,\\
G^{\mathcal{L}} &\,\,\text{if}\,\, 0^{\mathcal{L}} = \emptyset
\end{cases} 
\quad\eqlab{B}{=} G^{\mathcal{L}},
\end{eqnarray*}
where
(A) follows from (\ref{eq:r3xyoug24857}),
and (B) follows since $0^{\mathcal{L}} = \emptyset$.
Similarly, we obtain
$(0 \seq G)^{\mathcal{R}} = G^{\mathcal{R}}$.
This concludes $0 \seq G \cong G$.

(Proof of (ii))
We prove by induction on $\birth(G)$.
For the base case $G \cong 0$, we have
\begin{equation*}
(G \seq H) \seq J
\cong (0 \seq H) \seq J
\eqlab{A}\cong H \seq J
\eqlab{B} \cong 0 \seq (H \seq J)
\cong G \seq (H \seq J),
\end{equation*}
where (A) follows from (i) of this lemma,
and (B) follows from (i) of this lemma.
For the induction step for $G \not\cong 0$,
we have
\begin{eqnarray*}
((G \seq H)\seq J)^{\mathcal{L}}
&\eqlab{A}{=}& \begin{cases}
(G \seq H)^{\mathcal{L}} \seq J &\,\,\text{if}\,\, (G \seq H)^{\mathcal{L}} \neq \emptyset,\\
J^{\mathcal{L}} &\,\,\text{if}\,\, (G \seq H)^{\mathcal{L}} = \emptyset
\end{cases} \\
&\eqlab{B}{=}& \begin{cases}
(G^{\mathcal{L}} \seq H) \seq J &\,\,\text{if}\,\, (G \seq H)^{\mathcal{L}} \neq \emptyset, G^{\mathcal{L}} \neq \emptyset,\\
H^{\mathcal{L}} \seq J &\,\,\text{if}\,\, (G \seq H)^{\mathcal{L}} \neq \emptyset, G^{\mathcal{L}} = \emptyset,\\
J^{\mathcal{L}} &\,\,\text{if}\,\, (G \seq H)^{\mathcal{L}} = \emptyset
\end{cases} \\
&\eqlab{C}{=}& \begin{cases}
(G^{\mathcal{L}} \seq H) \seq J &\,\,\text{if}\,\, G^{\mathcal{L}} \neq \emptyset,\\
H^{\mathcal{L}} \seq J &\,\,\text{if}\,\, G^{\mathcal{L}} = \emptyset, H^{\mathcal{L}} \neq \emptyset,\\
J^{\mathcal{L}} &\,\,\text{if}\,\, G^{\mathcal{L}} = \emptyset, H^{\mathcal{L}} = \emptyset
\end{cases} \\
&\eqlab{D}{=}& \begin{cases}
G^{\mathcal{L}} \seq (H \seq J) &\,\,\text{if}\,\, G^{\mathcal{L}} \neq \emptyset,\\
H^{\mathcal{L}} \seq J &\,\,\text{if}\,\, G^{\mathcal{L}} = \emptyset, H^{\mathcal{L}} \neq \emptyset,\\
J^{\mathcal{L}} &\,\,\text{if}\,\, G^{\mathcal{L}} = \emptyset, H^{\mathcal{L}} = \emptyset
\end{cases} \\
&\eqlab{E}{=}& \begin{cases}
G^{\mathcal{L}} \seq (H \seq J) &\,\,\text{if}\,\, G^{\mathcal{L}} \neq \emptyset,\\
(H \seq J)^{\mathcal{L}} &\,\,\text{if}\,\, G^{\mathcal{L}} = \emptyset\\
\end{cases}\\
&\eqlab{F}{=}& (G \seq (H\seq J))^{\mathcal{L}},
\end{eqnarray*}
where
(A) follows from (\ref{eq:r3xyoug24857}),
(B) follows from (\ref{eq:r3xyoug24857}),
(C) follows since $(G \seq H)^{\mathcal{L}} \neq \emptyset$ is implied by $G^{\mathcal{L}} \neq \emptyset$,
and $(G \seq H)^{\mathcal{L}} = \emptyset$ is equivalent to $G^{\mathcal{L}} = H^{\mathcal{L}} = \emptyset$ by (\ref{eq:r3xyoug24857}),
(D) follows from the induction hypothesis,
(E) follows from (\ref{eq:r3xyoug24857}),
and (F) follows from (\ref{eq:r3xyoug24857}).
Similarly, we obtain
$((G \seq H)\seq J)^{\mathcal{R}} = (G \seq (H\seq J))^{\mathcal{R}}$.

(Proof of (iii))
We prove by induction on $\birth(G)$.
For the base case $G \cong 0$, we have
\begin{equation}
-(G \seq H)
\cong -(0 \seq H)
\eqlab{A}\cong -H
\eqlab{B}\cong 0 \seq (-H)
\eqlab{C}\cong (-0) \seq (-H)
\cong (-G) \seq (-H),
\end{equation}
where
(A) follows from (i) of this lemma,
(B) follows from (i) of this lemma,
and (C) follows from (\ref{eq:sppm6ebdbfgf}).
For the induction step for $G \not \cong 0$,  we have
\begin{eqnarray*}
(-(G \seq H))^{\mathcal{L}}
&\eqlab{A}{=}& -((G \seq H)^{\mathcal{R}}) \\
&\eqlab{B}{=}& \begin{cases}
-(G^{\mathcal{R}} \seq H) &\,\,\text{if}\,\,G^{\mathcal{R}} \neq \emptyset,\\
-(H^{\mathcal{R}}) &\,\,\text{if}\,\, G^{\mathcal{R}} = \emptyset\\
\end{cases} \\
&=& \begin{cases}
\left\{ -(G^{R} \seq H) : G^R \in G^{\mathcal{R}} \right\} &\,\,\text{if}\,\,G^{\mathcal{R}} \neq \emptyset,\\
-(H^{\mathcal{R}}) &\,\,\text{if}\,\, G^{\mathcal{R}} = \emptyset\\
\end{cases} \\
&\eqlab{C}=& \begin{cases}
\left\{ (-G^{R}) \seq (-H) : G^R \in G^{\mathcal{R}} \right\} &\,\,\text{if}\,\,G^{\mathcal{R}} \neq \emptyset,\\
-(H^{\mathcal{R}}) &\,\,\text{if}\,\, G^{\mathcal{R}} = \emptyset\\
\end{cases} \\
&{=}& \begin{cases}
(-(G^{\mathcal{R}})) \seq (-H) &\,\,\text{if}\,\,G^{\mathcal{R}} \neq \emptyset,\\
-(H^{\mathcal{R}})  &\,\,\text{if}\,\, G^{\mathcal{R}} = \emptyset\\
\end{cases} \\
&\eqlab{D}{=}& \begin{cases}
(-(G^{\mathcal{R}})) \seq (-H) &\,\,\text{if}\,\,G^{\mathcal{R}} \neq \emptyset,\\
(-H)^{\mathcal{L}} &\,\,\text{if}\,\, G^{\mathcal{R}} = \emptyset\\
\end{cases} \\
&=& \begin{cases}
(-(G^{\mathcal{R}})) \seq (-H) &\,\,\text{if}\,\,{-}(G^{\mathcal{R}}) \neq \emptyset,\\
(-H)^{\mathcal{L}} &\,\,\text{if}\,\, {-}(G^{\mathcal{R}}) = \emptyset\\
\end{cases} \\
&\eqlab{E}{=}& \begin{cases}
(-G)^{\mathcal{L}} \seq (-H) &\,\,\text{if}\,\,(-G)^{\mathcal{L}} \neq \emptyset,\\
(-H)^{\mathcal{L}} &\,\,\text{if}\,\, (-G)^{\mathcal{L}} = \emptyset\\
\end{cases} \\
&\eqlab{F}{=}&((-G) \seq (-H))^{\mathcal{L}},
\end{eqnarray*}
where
(A) follows from (\ref{eq:sppm6ebdbfgf}),
(B) follows from (\ref{eq:r3xyoug24857}),
(C) follows from the induction hypothesis,
(D) follows from (\ref{eq:sppm6ebdbfgf}),
(E) follows from (\ref{eq:sppm6ebdbfgf}),
and (F) follows from (\ref{eq:r3xyoug24857}).
Similarly, we obtain
$(-(G \seq H))^{\mathcal{R}} = ((-G) \seq (-H))^{\mathcal{R}}$.
\end{proof}

\subsection{Proof of Lemma \ref{lem:seq-neg}}
\label{subsec:proof-seq-neg}

 \begin{proof}[Proof of Lemma \ref{lem:seq-neg}]
  We consider the following three cases separately: the case $J \cong 0, a_0 \geq 1$, the case $J \cong \st, n \geq 1, a_0 = 0$, and the case $J < 0$.
 \begin{itemize}
\item The case $J \cong 0, a_0 \geq 1$: We have
\begin{eqnarray*}
 [a_n, a_{n-1}, \ldots, a_0] \seq J 
 &\cong&  [a_n, a_{n-1}, \ldots, a_0] \\
 &\eqlab{A}\cong& a_n \seq (-1) \seq a_{n-1} \seq (-1) \seq \cdots \seq (-1) \seq a_0 \seq \st \\
 &\eqlab{B}\cong& a_n \seq (-1) \seq a_{n-1} \seq (-1) \seq \cdots \seq (-1) \seq \llbracket a_0-1 \rrbracket \seq 1 \seq \st \\
 &\eqlab{C}<& 0,
\end{eqnarray*}
where
(A) follows from Definition \ref{def:block},
(B) follows from Theorem \ref{thm:int},
and (C) follows from $o(1 \seq \st) = \mathscr{R}$ and Lemma \ref{lem:seq-outcome}.

\item The case $J \cong \st, n \geq 1, a_0 = 0$: We have
\begin{eqnarray*}
 [a_n, a_{n-1}, \ldots, a_0] \seq J 
 &\cong&  [a_n, a_{n-1}, \ldots, a_1, 0] \seq \st\\
 &\eqlab{A}\cong& a_n \seq (-1) \seq a_{n-1} \seq (-1) \seq \cdots \seq (-1) \seq 0 \seq \st \seq \st \\
 &\cong& a_n \seq (-1) \seq a_{n-1} \seq (-1) \seq \cdots \seq (-1) \seq \st \seq \st \\
 &\eqlab{B}<& 0,
\end{eqnarray*}
where
(A) follows from Definition \ref{def:block},
and (B) follows from $o((-1) \seq \st \seq \st) = \mathscr{R}$ and Lemma \ref{lem:seq-outcome}.

\item The case $J < 0$: We obtain $o([a_n, a_{n-1}, \ldots, a_0] \seq J) = \mathscr{R}$ by $o(J) = \mathscr{R}$ and Lemma \ref{lem:seq-outcome}.
\end{itemize}
\end{proof}

\subsection{Proof of Lemma \ref{lem:ub9enb4w0s6a}}
\label{subsec:proof-ub9enb4w0s6a}

\begin{proof}[Proof of Lemma \ref{lem:ub9enb4w0s6a}]
(Proof of (i))
We consider the following three cases separately: the case $J \cong 0, a_0 \geq 1$, the case $J \cong \st, n \geq 1, a_0 = 0$, and the case $J < 0$.
\begin{itemize}
\item The case $J \cong 0, a_0 \geq 1$: We have
\begin{align*}
\lefteqn{\langle a_{n-1}, \ldots, a_0\rangle_{J} - \langle a_n, a_{n-1}, \ldots, a_0 \rangle_{J}}\\
&\eqlab{A}{=}  \left(\left(1 + \sum_{i = 0}^{n-1} a_i \right) \st + \sum_{i = 0}^{n-1} \left(1 + \sum_{j = i}^{n-1} a_j \right) \down^{i+1} - \down\right)
- \left(\left(1 + \sum_{i = 0}^n a_i \right) \st + \sum_{i = 0}^n \left(1 + \sum_{j = i}^n a_j \right) \down^{i+1} - \down\right)\\
&= a_n \st + \sum_{i = 0}^{n} a_n \up^{i+1} + \up^{n+1} \\
&\eqlab{B}{\not\leq} 0,
\end{align*}
where
(A) follows from the first case of (\ref{eq:xhzxoe0bbtgq}),
and (B) follows from Propositions \ref{prop:uptimal-comp} and \ref{prop:uptimal-star}.

\item The case $J \cong \st, n \geq 1, a_0 = 0$: We have
\begin{align*}
\lefteqn{\langle a_{n-1}, \ldots, a_0\rangle_{J} - \langle a_n, a_{n-1}, \ldots, a_0 \rangle_{J}}\\
&\eqlab{A}{=}  \left(\left(\sum_{i = 1}^{n-1} a_i \right) \st + \sum_{i = 1}^{n-1} \left(1 + \sum_{j = i}^{n-1} a_j \right) \down^{i}\right)
- \left(\left(\sum_{i = 1}^{n} a_i \right) \st + \sum_{i = 1}^{n} \left(1 + \sum_{j = i}^{n} a_j \right) \down^{i}\right)\\
&= a_n \st + \sum_{i = 1}^{n} a_n \up^{i} +  \up^{n} \\
&\eqlab{B}{\not\leq} 0,
\end{align*}
where
(A) follows from the second case of (\ref{eq:xhzxoe0bbtgq}),
and (B) follows from Propositions \ref{prop:uptimal-comp} and \ref{prop:uptimal-star}.

\item The case $J < 0$: We have
\begin{align*}
\lefteqn{\langle a_{n-1}, \ldots, a_0\rangle_{J} - \langle a_n, a_{n-1}, \ldots, a_0 \rangle_{J}}\\
&\eqlab{A}{=} \left(J +  \left(1+\sum_{i = 0}^{n-1} a_i \right) \st + \sum_{i = 1}^k \left(1+ \sum_{j = 0}^{n-1} a_j \right) \down^i + \sum_{i = 1}^{n-1} \left(1 + \sum_{j = i}^{n-1} a_j \right) \down^{k+i} \right)\\
&\quad - \left( J + \left(1+\sum_{i = 0}^n a_i \right) \st + \sum_{i = 1}^k \left(1+ \sum_{j = 0}^n a_j \right)  \down^i + \sum_{i = 1}^n \left(1 + \sum_{j = i}^n a_j \right) \down^{k+i} \right)\\
&= a_n \st + \sum_{i = 1}^{n+k} a_n \up ^i + \up^{n+k}\\
&\eqlab{B}{\not\leq} 0,
\end{align*}
where
(A) follows from the third case of (\ref{eq:xhzxoe0bbtgq}),
and (B) follows from Propositions \ref{prop:uptimal-comp} and \ref{prop:uptimal-star}.
\end{itemize}

(Proof of (ii))
We consider the following two cases separately: the case $J \cong 0, a_0 \geq 1$ and the case $J < 0$.
Note that the case $J \cong \st, n \geq 1, a_0 = 0$ is excluded by the assumption $n = 0$.
\begin{itemize}
\item The case $J \cong 0, a_0 \geq 1$: We have
\begin{align*}
J - \langle a_0 \rangle_{J}
\eqlab{A}{=}  0 - \left(\left(1 + a_0 \right) \st + \left(1 + a_0 \right) \down - \down\right)
= \left(1 + a_0 \right) \st + a_0 \up
\eqlab{B}{\not\leq} 0,
\end{align*}
where
(A) follows from the first case of (\ref{eq:xhzxoe0bbtgq}),
and (B) follows from Propositions \ref{prop:uptimal-comp} and \ref{prop:uptimal-star}.

\item The case $J < 0$: We have
\begin{align*}
J - \langle a_0\rangle_{J}
\eqlab{A}{=} J - \left( J + \left(1+ a_0\right) \st + \sum_{i = 1}^k \left(1+ a_0 \right)  \down^i\right)
=  \left(1+a_0 \right)\st + \sum_{i = 1}^k \left(1+ a_0 \right) \up^i
\eqlab{B}{\not\leq} 0,
\end{align*}
where
(A) follows from the third case of (\ref{eq:xhzxoe0bbtgq}),
and (B) follows from Propositions \ref{prop:uptimal-comp} and \ref{prop:uptimal-star}.
\end{itemize}
\end{proof}

\subsection{Proof of Lemma \ref{lem:3rp1dfrhr0nc}}

 \begin{proof}[Proof of Lemma \ref{lem:3rp1dfrhr0nc}]
 (Proof of (i))
 Choose a Left option $\langle a_n, a_{n-1}, \ldots, a_0 \rangle_{J}^L$ of $\langle a_n, a_{n-1}, \ldots, a_0 \rangle_{J}$ arbitrarily.
 We consider the following three cases separately: the case $J \cong 0, a_0 \geq 1$, the case $J \cong \st, n \geq 1, a_0 = 0$, and the case $J < 0$.
\begin{itemize}
\item The case $J \cong 0, a_0 \geq 1$: We have
\begin{align*}
\langle a_n, a_{n-1}, \ldots, a_0 \rangle_{J}^L
&\eqlab{A}\leq \langle a_n, a_{n-1}, \ldots, a_0 \rangle_{J} - \sum_{i = 1}^{\deg(\langle a_n, a_{n-1}, \ldots, a_0 \rangle_{J})} \down^i + \st\\
&\eqlab{B}= \left(1 + \sum_{i = 0}^n a_i \right) \st + \sum_{i = 0}^n \left(1 + \sum_{j = i}^n a_j \right) \down^{i+1} - \down - \sum_{i = 1}^{n+1} \down^i + \st \\
&= \left(\sum_{i = 0}^n a_i\right) \st + \sum_{i = 0}^n \left(\sum_{j = i}^n a_j \right) \down^{i+1} - \down\\
&\eqlab{C}=\left(\sum_{i = 0}^p a_i\right) \st + \sum_{i = 0}^n \left(\sum_{j = i}^p a_j \right) \down^{i+1} - \down\\
&\eqlab{D}=\left(\sum_{i = 0}^p a_i\right) \st + \sum_{i = 0}^p \left(\sum_{j = i}^p a_j \right) \down^{i+1} - \down\\
&=\left(1 + \sum_{i = 0}^{p-1} a_i + (a_p-1) \right) \st + \sum_{i = 0}^p \left(1 + \sum_{j = i}^{p-1} a_j  + (a_p-1) \right) \down^{i+1} - \down\\
&\eqlab{E}= \langle a_p-1, a_{p-1}, \ldots, a_0 \rangle_{J},
\end{align*}
where
(A) follows from Lemma \ref{lem:uptimal-option} (i) and (\ref{lem:seq-deg}),
(B) follows from the first case of (\ref{eq:xhzxoe0bbtgq}),
(C) follows from $a_n = a_{n-1} = \cdots = a_{p+1} = 0$,
(D) follows since $\sum_{j = i}^p a_j = 0$ for $i > p$,
and (E) follows from the first case of (\ref{eq:xhzxoe0bbtgq}).

\item The case $J \cong \st, n \geq 1$, $a_0 = 0$: We have
\begin{align*}
\langle a_n, a_{n-1}, \ldots, a_0 \rangle_{J}^L
&\eqlab{A}\leq \langle a_n, a_{n-1}, \ldots, a_0 \rangle_{J} - \sum_{i = 1}^{\deg(\langle a_n, a_{n-1}, \ldots, a_0 \rangle_{J})} \down^i + \st\\
&\eqlab{B}= \left(\sum_{i = 1}^n a_i \right) \st + \sum_{i = 1}^n \left(1 + \sum_{j = i}^n a_j \right) \down^i - \sum_{i = 1}^{n} \down^i + \st\\
&= \left(\sum_{i = 1}^n a_i - 1\right) \st + \sum_{i = 1}^n \left(\sum_{j = i}^n a_j \right) \down^i\\
&\eqlab{C}= \left(\sum_{i = 1}^p a_i - 1\right) \st + \sum_{i = 1}^n \left(\sum_{j = i}^p a_j \right) \down^i\\
&\eqlab{D}= \left(\sum_{i = 1}^p a_i - 1\right) \st + \sum_{i = 1}^p \left(\sum_{j = i}^p a_j \right) \down^i\\
&= \left(\sum_{i = 1}^{p-1} a_i + (a_p- 1)\right) \st + \sum_{i = 1}^p \left(\sum_{j = i}^{p-1} a_j + (a_p-1) \right) \down^i\\
&\eqlab{E}= \langle a_p-1, a_{p-1}, \ldots, a_0 \rangle_{J},
\end{align*}
where
(A) follows from Lemma \ref{lem:uptimal-option} (i) and (\ref{lem:seq-deg}),
(B) follows from the second case of (\ref{eq:xhzxoe0bbtgq}),
(C) follows from $a_n = a_{n-1} = \cdots = a_{p+1} = 0$,
(D) follows since $\sum_{j = i}^p a_j = 0$ for $i > p$,
and (E) follows from the second case of (\ref{eq:xhzxoe0bbtgq}).

\item The case $J < 0$: We have
\begin{align*}
\langle a_n, a_{n-1}, \ldots, a_0 \rangle_{J}^L
&\eqlab{A}\leq \langle a_n, a_{n-1}, \ldots, a_0 \rangle_{J} - \sum_{i = 1}^{\deg(\langle a_n, a_{n-1}, \ldots, a_0 \rangle_{J})} \down^i + \st\\
&\eqlab{B}=  J + \left(1+\sum_{i = 0}^n a_i \right) \st + \sum_{i = 1}^k \left(1+ \sum_{j = 0}^n a_j \right)  \down^i + \sum_{i = 1}^n \left(1 + \sum_{j = i}^n a_j \right) \down^{k+i} - \sum_{i = 1}^{n+k} \down^i + \st \\
&=  J + \left(\sum_{i = 0}^n a_i \right) \st + \sum_{i = 1}^k \left(\sum_{j = 0}^n a_j \right)  \down^i + \sum_{i = 1}^n \left(\sum_{j = i}^n a_j \right) \down^{k+i}\\
&\eqlab{C}=  J + \left(\sum_{i = 0}^p a_i \right) \st + \sum_{i = 1}^k \left(\sum_{j = 0}^p a_j \right)  \down^i + \sum_{i = 1}^n \left(\sum_{j = i}^p a_j \right) \down^{k+i}\\
&\eqlab{D}=  J + \left(\sum_{i = 0}^p a_i \right) \st + \sum_{i = 1}^k \left(\sum_{j = 0}^p a_j \right)  \down^i + \sum_{i = 1}^p \left(\sum_{j = i}^p a_j \right) \down^{k+i}\\
&\eqlab{E}{=} \langle a_p-1, a_{p-1}, \ldots, a_0 \rangle_{J},
\end{align*}
where
(A) follows from Lemma \ref{lem:uptimal-option} (i) and (\ref{lem:seq-deg}),
(B) follows from the third case of (\ref{eq:xhzxoe0bbtgq}),
(C) follows from $a_n = a_{n-1} = \cdots = a_{p+1} = 0$,
(D) follows since $\sum_{j = i}^p a_j = 0$ for $i > p$,
and (E) follows from the third case of (\ref{eq:xhzxoe0bbtgq}).
\end{itemize}

(Proof of (ii))
 Choose a Left option $\langle a_n, a_{n-1}, \ldots, a_0 \rangle_{J}^L$ of $\langle a_n, a_{n-1}, \ldots, a_0 \rangle_{J}$ arbitrarily.
  We consider the following two cases separately: the case $J \cong \st, n \geq 1, a_0 = 0$ and the case $J < 0$.
  Note that the case $J \cong 0, a_0 \geq 1$ is excluded by the assumption $a_n = a_{n-1} = \cdots = a_0 = 0$.
\begin{itemize}

\item The case $J \cong \st$, $n \geq 1$, $a_0 = 0$: We have
\begin{align*}
\langle a_n, a_{n-1}, \ldots, a_0 \rangle_{J}^L
&\eqlab{A}\leq \langle a_n, a_{n-1}, \ldots, a_0 \rangle_{J} - \sum_{i = 1}^{\deg(\langle a_n, a_{n-1}, \ldots, a_0 \rangle_{J})} \down^i + \st\\
&\eqlab{B}= \left(\sum_{i = 1}^n a_i \right) \st + \sum_{i = 1}^n \left(1 + \sum_{j = i}^n a_j \right) \down^i - \sum_{i = 1}^{n} \down^i + \st\\
&= \left(\sum_{i = 1}^n a_i + 1\right) \st + \sum_{i = 1}^n \left(\sum_{j = i}^n a_j \right) \down^i\\
&\eqlab{C}= \st\\
&= J,
\end{align*}
where
(A) follows from Lemma \ref{lem:uptimal-option} (i) and (\ref{lem:seq-deg}),
(B) follows from the second case of (\ref{eq:xhzxoe0bbtgq}),
and (C) follows from $a_n = a_{n-1} = \cdots = a_0 = 0$.

\item The case $J < 0$: We have
\begin{align*}
\langle a_n, a_{n-1}, \ldots, a_0 \rangle_{J}^L
&\eqlab{A}\leq \langle a_n, a_{n-1}, \ldots, a_0 \rangle_{J} - \sum_{i = 1}^{\deg(\langle a_n, a_{n-1}, \ldots, a_0 \rangle_{J})} \down^i + \st\\
&\eqlab{B}= J + \left(1+\sum_{i = 0}^n a_i \right) \st + \sum_{i = 1}^k \left(1+ \sum_{j = 0}^n a_j \right)  \down^i + \sum_{i = 1}^n \left(1 + \sum_{j = i}^n a_j \right) \down^{k+i}  - \sum_{i = 1}^{n+k} \down^i + \st\\
&\eqlab{C}= J + \st + \sum_{i = 1}^k \down^i + \sum_{i = 1}^n \down^{k+i}  - \sum_{i = 1}^{n+k} \down^i + \st\\
&= J,
\end{align*}
where
(A) follows from Lemma \ref{lem:uptimal-option} (i) and (\ref{lem:seq-deg}),
(B) follows from the third case of (\ref{eq:xhzxoe0bbtgq}),
and (C) follows from $a_n = a_{n-1} = \cdots = a_0 = 0$.
\end{itemize}
\end{proof}

\subsection{Proof of Lemma \ref{lem:q7g997f5n2b0}}

\begin{proof}[Proof of Lemma \ref{lem:q7g997f5n2b0}]
(Proof of (i))
  We consider the following three cases separately: the case $J \cong 0, a_0 \geq 1$, the case $J \cong \st, n \geq 1, a_0 = 0$, and the case $J < 0$.
\begin{itemize}
\item The case $J \cong 0, a_0 \geq 1$: We have
\begin{align*}
\lefteqn{\langle a_p-1, a_{p-1}, a_{p-2}, \ldots, a_0 \rangle_{J} - \langle \underbrace{0, \ldots, 0}_{n-p}, a_p, a_{p-1}, \ldots, a_0 \rangle_{J}}\\
&\eqlab{A}= \left(\left(\sum_{i = 0}^p a_i \right) \st + \sum_{i = 0}^p \left(\sum_{j = i}^p a_j \right) \down^{i+1} - \down\right)
 - \left(\left(1 + \sum_{i = 0}^p a_i \right) \st + \sum_{i = 0}^n \left(1 + \sum_{j = i}^p a_j \right) \down^{i+1} - \down\right)\\
&= \st + \sum_{i = 0}^p \up^{i+1} + \sum_{i = p+1}^n \left(1 + \sum_{j = i}^p a_j \right) \up^{i+1}\\
&\eqlab{B}= \st + \sum_{i = 0}^p \up^{i+1} + \sum_{i = p+1}^n \up^{i+1}\\
&= \st + \sum_{i = 1}^{n+1} \up^i\\
&\eqlab{C}\parallel 0,
\end{align*}
where
(A) follows from the first case of (\ref{eq:xhzxoe0bbtgq}),
(B) follows since $\sum_{j = i}^p a_j = 0$ for $i > p$,
and (C) follows from Proposition \ref{prop:uptimal-star}.

\item The case $J \cong \st, n \geq 1$, $a_0 = 0$: We have
\begin{align*}
\lefteqn{\langle a_p-1, a_{p-1}, a_{p-2}, \ldots, a_0 \rangle_{J} - \langle \underbrace{0, \ldots, 0}_{n-p}, a_p, a_{p-1}, \ldots, a_0 \rangle_{J}}\\
&\eqlab{A}= \left(\left(\sum_{i = 1}^p a_i - 1 \right) \st + \sum_{i = 1}^p \left(\sum_{j = i}^p a_j \right) \down^i \right)
 - \left(\left(\sum_{i = 1}^p a_i \right) \st + \sum_{i = 1}^n \left(1 + \sum_{j = i}^p a_j \right) \down^i \right)\\
&= \st + \sum_{i = 1}^p \up^i + \sum_{i = p+1}^n \left(1 + \sum_{j = i}^p a_j \right) \up^i\\
&\eqlab{B}= \st + \sum_{i = 1}^p \up^{i} + \sum_{i = p+1}^n \up^i\\
&= \st + \sum_{i = 1}^n \up^i\\
&\eqlab{C}\parallel 0,
\end{align*}
where
(A) follows from the second case of (\ref{eq:xhzxoe0bbtgq}),
(B) follows since $\sum_{j = i}^p a_j = 0$ for $i > p$,
and (C) follows from Proposition \ref{prop:uptimal-star}.

\item The case $J < 0$: We have
\begin{align*}
\lefteqn{\langle a_p-1, a_{p-1}, a_{p-2}, \ldots, a_0 \rangle_{J} - \langle \underbrace{0, \ldots, 0}_{n-p}, a_p, a_{p-1}, \ldots, a_0 \rangle_{J}}\\
&\eqlab{A}= \left(J + \left(\sum_{i = 0}^p a_i \right) \st + \sum_{i = 1}^k \left(\sum_{j = 0}^p a_j \right) \down^i + \sum_{i = 1}^p \left(\sum_{j = i}^p a_j \right) \down^{k+i}\right)\\
&\quad -\left(J + \left(1+\sum_{i = 0}^p a_i \right) \st + \sum_{i = 1}^k  \left(1+ \sum_{j = 0}^p a_j \right)  \down^i + \sum_{i = 1}^n \left(1 + \sum_{j = i}^p a_j \right) \down^{k+i}\right)\\
&= \st + \sum_{i = 1}^k \up^{i} + \sum_{i = 1}^p \up^{k+i} +  \sum_{i = p+1}^n \left(1 + \sum_{j = i}^p a_j \right) \up^{k+i}\\
&\eqlab{B}= \st + \sum_{i = 1}^k \up^{i} + \sum_{i = 1}^p \up^{k+i} +  \sum_{i = p+1}^n \up^{k+i}\\
&= \st +  \sum_{i = 1}^{k+n} \up^i\\
&\eqlab{C}\parallel 0,
\end{align*}
where
(A) follows from the third case of (\ref{eq:xhzxoe0bbtgq}),
(B) follows since $\sum_{j = i}^p a_j = 0$ for $i > p$,
and (C) follows from Proposition \ref{prop:uptimal-star}.
\end{itemize}

(Proof of (ii))
  We consider the following two cases separately: the case $J \cong \st, n \geq 1, a_0 = 0$, and the case $J < 0$.
  Note that the case $J \cong 0, a_0 \geq 1$ is excluded by the assumption $a_n = a_{n-1} = \cdots = a_0 = 0$.
  
\begin{itemize}
\item The case $J \cong \st, n \geq 1, a_0 = 0$: We have
\begin{align*}
J - \langle \underbrace{0, \ldots, 0}_{n+1} \rangle_{J}
\eqlab{A}= J + \sum_{i = 1}^n \up^i
= \st + \sum_{i = 1}^n \up^i
\eqlab{B}\parallel 0,
\end{align*}
where
(A) follows from the second case of (\ref{eq:xhzxoe0bbtgq}),
and (B) follows from Proposition \ref{prop:uptimal-star}.

\item The case $J < 0$: We have
\begin{align*}
J - \langle \underbrace{0, \ldots, 0}_{n+1} \rangle_{J}
\eqlab{A}= \st + \sum_{i = 1}^k \up^{i} + \sum_{i = 1}^n \up^{k+i}
= \st +  \sum_{i = 1}^{k+n} \up^i
\eqlab{B}\parallel 0,
\end{align*}
where
(A) follows from the third case of (\ref{eq:xhzxoe0bbtgq}),
and (B) follows from Proposition \ref{prop:uptimal-star}.
\end{itemize}
\end{proof}

\subsection{Proof of Lemma \ref{lem:e1xr1ufi5spj}}
\label{subsec:proof-e1xr1ufi5spj}

\begin{proof}[Proof of Lemma \ref{lem:e1xr1ufi5spj}]
(Proof of (i))
We consider the following three cases separately: the case $J \cong 0$, the case $J \cong \st$, and the case $J < 0$.
\begin{itemize}
\item The case $J \cong 0$: We have
\begin{align*}
\langle 0, a_{n-1}. \ldots, a_0 \rangle_{J}
&\eqlab{A}= \left(1 + \sum_{i = 0}^{n-1} a_i \right) \st + \sum_{i = 0}^n \left(1 + \sum_{j = i}^{n-1} a_j \right) \down^{i+1} - \down\\
&= \left(1 + \sum_{i = 0}^{n-1} a_i \right) \st + \sum_{i = 0}^{n-1} \left(1 + \sum_{j = i}^{n-1} a_j \right) \down^{i+1} - \down + \down^{n+1}\\
&\eqlab{B}= \langle a_{n-1}, \ldots, a_0 \rangle_{J} + \down^{n+1}\\
&= \langle a_{n-1}, \ldots, a_0 \rangle_{J} + \down^{d},
\end{align*}
where
(A) follows from the first case of (\ref{eq:xhzxoe0bbtgq}),
and (B) follows from the first case of (\ref{eq:xhzxoe0bbtgq}).

\item The case $J \cong \st$: We have
\begin{align*}
\langle 0, a_{n-1}, \ldots, a_0 \rangle_{J}
&\eqlab{A}= \left(\sum_{i = 1}^{n-1} a_i \right) \st + \sum_{i = 1}^{n} \left(1 + \sum_{j = i}^{n-1} a_j \right) \down^{i}\\
&= \left(\sum_{i = 1}^{n-1} a_i \right) \st + \sum_{i = 1}^{n-1} \left(1 + \sum_{j = i}^{n-1} a_j \right) \down^{i} + \down^n\\
&\eqlab{B}= \langle a_{n-1}, \ldots, a_0 \rangle_{J} + \down^{n}\\
&= \langle a_{n-1}, \ldots, a_0 \rangle_{J} + \down^{d},
\end{align*}
where
(A) follows from the second case of (\ref{eq:xhzxoe0bbtgq}),
and (B) follows from the second case of (\ref{eq:xhzxoe0bbtgq}).

\item The case $J < 0$: We have
\begin{align*}
\lefteqn{\langle 0, a_{n-1}, \ldots, a_0\rangle_{J}}\\
&\eqlab{A}{=} J + \left(1+\sum_{i = 0}^{n-1} a_i \right) \st + \sum_{i = 1}^k \left(1+ \sum_{j = 0}^{n-1} a_j \right)  \down^i + \sum_{i = 1}^n \left(1 + \sum_{j = i}^{n-1} a_j \right) \down^{k+i}\\
&= J + \left(1+\sum_{i = 0}^{n-1} a_i \right) \st + \sum_{i = 1}^k \left(1+ \sum_{j = 0}^{n-1} a_j \right)  \down^i + \sum_{i = 1}^{n-1} \left(1 + \sum_{j = i}^{n-1} a_j \right) \down^{k+i} + \left(1 + \sum_{j = n}^{n-1} a_j \right) \down^{k+n}\\
&\eqlab{B}= J + \left(1+\sum_{i = 0}^{n-1} a_i \right) \st + \sum_{i = 1}^k \left(1+ \sum_{j = 0}^{n-1} a_j \right)  \down^i + \sum_{i = 1}^{n-1} \left(1 + \sum_{j = i}^{n-1} a_j \right) \down^{k+i} + \down^{k+n}\\
&\eqlab{C}= \langle a_{n-1}, \ldots, a_0 \rangle_{J} + \down^{k+n}\\
&= \langle a_{n-1}, \ldots, a_0 \rangle_{J} + \down^{d},
\end{align*}
where
(A) follows from the third case of (\ref{eq:xhzxoe0bbtgq}),
(B) follows from $\sum_{j = n}^{n-1} a_j = 0$ since $n > (n-1)$,
and (C) follows from the third case of (\ref{eq:xhzxoe0bbtgq}).
\end{itemize}
\end{proof}

\subsection{Proof of Lemma \ref{lem:sirsnip1kywb}}
\label{subsec:proof-sirsnip1kywb}

\begin{proof}[Proof of Lemma \ref{lem:sirsnip1kywb}]
The sequence $(q'', p''_1, p''_2, \ldots, p''_d)$ is
\begin{eqnarray}
q'' &=& q', \nonumber\\
p''_i &=&  \begin{cases}
p'_i-1 &\,\,\text{if}\,\,1 \leq i \leq d',\\
0 &\,\,\text{if}\,\,d'+1 \leq i \leq d
\end{cases}
\quad\,\,\text{for}\,\, i = 1, 2, \ldots, d.
\label{eq:tyoty5gmj38s}
\end{eqnarray}

The condition (\ref{eq:4zvhmrup5k1l}) holds because
\begin{eqnarray*}
q'' \st + \sum_{i = 1}^{d} p''_i \down^i 
\eqlab{A}\cong q' \st + \sum_{i = 1}^{d} (p'_i-1) \down^i 
= q' \st + \sum_{i = 1}^{d} p'_i \down^i  - \sum_{i = 1}^{d'} \down^i
> q' \st + \sum_{i = 1}^{d} p'_i \down^i
\eqlab{B}\geq \langle a_n, a_{n-1}, \ldots, a_0\rangle_{J},
\end{eqnarray*}
where
(A) follows from (\ref{eq:tyoty5gmj38s}),
and (B) follows from (\ref{eq:4zvhmrup5k1l}).

The condition (\ref{eq:lpwhcjc3w1re}) holds because for any $i = 1 \ldots, d$, we have
\begin{eqnarray}
\label{eq:sutnlc773ymhf}
p_i \eqlab{A}\geq p'_i
\geq \begin{cases}
p'_i-1 &\,\,\text{if}\,\,1 \leq i \leq d',\\
0 &\,\,\text{if}\,\,d'+1 \leq i \leq d
\end{cases}
\eqlab{B}= p''_i ,
\end{eqnarray}
where (A) follows from (\ref{eq:lpwhcjc3w1re}),
and (B) follows from (\ref{eq:tyoty5gmj38s}).

The condition (\ref{eq:sl195b7g7f7t}) is confirmed as follows.
We have
\begin{equation}
p'_1 \eqlab{A}\geq p'_2 \eqlab{A}\geq \cdots \eqlab{A}\geq p'_{d'} \eqlab{B}\geq 1,
\end{equation}
where
(A)s follow from (\ref{eq:sl195b7g7f7t}),
and (B) follows from $d' = \deg(H')$.
This implies
\begin{eqnarray}
p'_1-1 \geq p'_2-1 \geq \cdots \geq p'_{d'}-1 \geq 0
&\eqlab{A}\implies& p''_1 \geq p''_2 \geq \cdots \geq p''_{d'} \geq  0 = p''_{d'+1} = \cdots = p''_d\nonumber\\
&\implies& p''_1 \geq p''_2 \geq \cdots \geq p''_{d} \label{eq:ahhsnkb31j0g}
\end{eqnarray}
as desired,
where (A) follows from (\ref{eq:tyoty5gmj38s}).
\end{proof}

\subsection{Proof of Lemma \ref{lem:qdnup6wlyl8k}}
\label{subsec:proof-qdnup6wlyl8k}

\begin{proof}[Proof of Lemma \ref{lem:qdnup6wlyl8k}]
We divide into the following two cases separately: the case $p'_{d'} \geq 2$ and the case $p'_{d'} = 1$.
\begin{itemize}
\item The case $p'_{d'} \geq 2$: Then the sequence $(q'', p''_1, p''_2, \ldots, p''_d)$ is
\begin{eqnarray}
q'' &=& q', \nonumber\\
p''_i &=&  \begin{cases}
p'_i-1 &\,\,\text{if}\,\,1 \leq i \leq d'-1,\\
p'_i-2 &\,\,\text{if}\,\,i = d',\\
0 &\,\,\text{if}\,\,d'+1 \leq i \leq d
\end{cases}
\quad\,\,\text{for}\,\, i = 1, 2, \ldots, d.
\label{eq:jp6kwljig7cg}
\end{eqnarray}

The condition (\ref{eq:4zvhmrup5k1l}) is confirmed as
\begin{eqnarray*}
q'' \st + \sum_{i = 1}^{d} p''_i \down^i
&\eqlab{A}\cong& q' \st + \sum_{i = 1}^{d'-1} (p'_i-1) \down^i + (p'-2) \down^{d'}\\
&=& q' \st + \sum_{i = 1}^{d'} p'_i \down^i - \left(\sum_{i = 1}^{d'-1} \down^i + 2 \down^{d'}\right)\\
&>& q' \st + \sum_{i = 1}^{d'} p'_i \down^i\\
&\eqlab{B}\geq& \langle a_n, a_{n-1}, \ldots, a_0 \rangle_{J},
\end{eqnarray*}
where 
(A) follows from (\ref{eq:jp6kwljig7cg}),
and (B) follows from (\ref{eq:4zvhmrup5k1l}).

The condition (\ref{eq:lpwhcjc3w1re}) holds because for any $i = 1, 2, \ldots, d$, we have
\begin{eqnarray*}
p_i \eqlab{A}\geq p'_i
\geq \begin{cases}
p'_i-1 &\,\,\text{if}\,\,1 \leq i \leq d'-1,\\
p'_i-2 &\,\,\text{if}\,\,i = d',\\
0 &\,\,\text{if}\,\,d'+1 \leq i \leq d
\end{cases}
\eqlab{B}= p''_i,
\end{eqnarray*}
where
(A) follows from (\ref{eq:lpwhcjc3w1re}),
and (B) follows from (\ref{eq:jp6kwljig7cg}).

The condition (\ref{eq:sl195b7g7f7t}) is confirmed as follows.
We have
\begin{equation}
p'_1 \eqlab{A}\geq p'_2 \eqlab{A} \geq \cdots \eqlab{A}\geq p'_{d'} \geq 2,
\end{equation}
where (A)s follow from (\ref{eq:sl195b7g7f7t}).
This implies
\begin{eqnarray*}
\lefteqn{p'_1-1 \geq p'_2-1 \geq \cdots \geq p'_{d'-1}-1 \geq p'_{d'}-1 \geq 1}\\
&\implies& p'_1-1 \geq p'_2-1 \geq \cdots \geq p'_{d'-1}-1 \geq p'_{d'}-2 \geq 0 \\
&\eqlab{A}\implies& p''_1 \geq p''_2 \geq \cdots \geq p''_{d'} \geq  0 = p''_{d'+1} = \cdots = p''_d\\
&\implies& p''_1 \geq p''_2 \geq \cdots \geq p''_{d}
\end{eqnarray*}
as desired,
where (A) follows from (\ref{eq:jp6kwljig7cg}).

\item  The case $p'_{d'} = 1$: Then the sequence $(q'', p''_1, p''_2, \ldots, p''_d)$ is
\begin{eqnarray}
q'' &=& q',\nonumber\\
p''_i &=&  \begin{cases}
p'_i-1 &\,\,\text{if}\,\,1 \leq i \leq d',\\
0 &\,\,\text{if}\,\,d'+1 \leq i \leq d
\end{cases}
\quad\,\,\text{for}\,\, i = 1, 2, \ldots, d. \label{eq:uv7sfqz8xddv}
\end{eqnarray}

The condition (\ref{eq:4zvhmrup5k1l}) is confirmed as
\begin{eqnarray*}
\lefteqn{\langle a_n, a_{n-1}, \ldots, a_0 \rangle_{J} - \left(q'' \st + \sum_{i = 1}^{d} p''_i \down^i \right)}\\
&\cong&  \left(q \st + \sum_{i = 1}^{d} p_i \down^i \right) - \left(q'' \st + \sum_{i = 1}^{d} p''_i \down^i \right)\\
&\eqlab{A}\cong&  \left(q \st + \sum_{i = 1}^{d} p_i \down^i \right) - \left(q' \st + \sum_{i = 1}^{d'} (p'_i-1) \down^i \right)\\
&=&  \left(q \st + \sum_{i = 1}^{d} p_i \down^i \right) - \left(q' \st + \sum_{i = 1}^{d'} p'_i \down^i \right)  + \down^{d'} + \sum_{i = 1}^{d'-1} \down^i + \st\\
&\eqlab{B}{\leq}&  \left(q \st + \sum_{i = 1}^{d'} p_i \down^i \right) - \left(q' \st + \sum_{i = 1}^{d'} p'_i \down^i \right)  + \down^{d'} + \sum_{i = 1}^{d'-1} \down^i + \st\\
&=& (q+q'+1) \st + \sum_{i = 1}^{d'-1}(p_i - p'_i) \down^i + (p_{d'} - p'_{d'} + 1) \down^{d'} +  \sum_{i = 1}^{d'-1} \down^i\\
&\eqlab{C}{\leq}& (q+q'+1) \st + p_{d'} \down^{d'} +  \sum_{i = 1}^{d'-1} \down^i\\
&\eqlab{D}\leq& (q+q'+1) \st + p_d \down^{d'} +  \sum_{i = 1}^{d'-1} \down^i\\
&\eqlab{E}{\leq}& (q+q'+1) \st + 2 \down^{d'} +  \sum_{i = 1}^{d'-1} \down^i\\
&\eqlab{F}{<}& 0,
\end{eqnarray*}
where
(A) follows from (\ref{eq:uv7sfqz8xddv}),
(B) follows from $d' \leq d$,
(C) follows from $p_i \geq p'_i$ and $p'_{d'} = 1$,
(D) follows since $p_{d'} \geq p_d$ by $d' \leq d$ and (\ref{eq:lpwhcjc3w1re}),
(E) follows from the assumption $p_d \geq 2$,
and (F) follows from Proposition \ref{prop:uptimal-star}.

The conditions (\ref{eq:lpwhcjc3w1re}) and  (\ref{eq:sl195b7g7f7t}) are confirmed as (\ref{eq:sutnlc773ymhf}) and (\ref{eq:ahhsnkb31j0g}), respectively.
\end{itemize}
\end{proof}

\bibliographystyle{plain}
\bibliography{arxiv}

\end{document}